\documentclass{article}

\usepackage[noadjust]{cite}

\usepackage{amsmath,amssymb,amsfonts,amsthm}
\usepackage{threeparttable}
\usepackage{adjustbox}
\usepackage{makecell}
\usepackage{booktabs}
\usepackage{nccmath,mathtools}
\usepackage{subfig}
\usepackage{geometry}
\usepackage[hidelinks]{hyperref}
\usepackage{tablefootnote}

\usepackage{enumitem}
\usepackage{bbold}
\setlist{topsep=1pt,itemsep=0pt,labelwidth=8pt,labelsep=8pt,leftmargin=\parindent}
\usepackage{graphicx}

\usepackage[ruled]{algorithm2e}
\let\oldnl\nl \newcommand{\nonl}{\renewcommand{\nl}{\let\nl\oldnl}}
\usepackage{setspace}

\DeclareMathOperator*{\argmin}{arg\,min}

\DeclarePairedDelimiterX\setc[2]{\{}{\}}{#1 : #2}
\DeclarePairedDelimiterX\setv[2]{\{}{\}}{#1 \;\delimsize\vert\; #2}
\DeclarePairedDelimiterX\Econd[2]{[}{]}{#1 \mkern2mu\delimsize\vert\mkern2mu #2}
\DeclarePairedDelimiterX\Vcond[2]{(}{)}{#1 \mkern2mu\delimsize\vert\mkern2mu #2}
\newtheorem{lemma}{Lemma}
\newtheorem{theorem}{Theorem}
\newtheorem{corollary}{Corollary}

\theoremstyle{definition}
\newtheorem{definition}{Definition}
\newtheorem{assumption}{Assumption}

\theoremstyle{remark}
\newtheorem{remark}{Remark}

\newcommand\numberthis{\addtocounter{equation}{1}\tag{\theequation}}

\newcommand{\defeq}{\overset{\mathrm{def}}{=\joinrel=}}
\DeclareMathOperator{\tr}{tr}

\newcommand{\AlgFullname}{Zeroth-Order Incremental Variance Reduction}
\newcommand{\AlgAbbrv}{ZIVR}

\title{\LARGE \bf
\AlgAbbrv: An Incremental Variance Reduction Technique For Zeroth-Order Composite Problems
}

\author{Silan Zhang and Yujie Tang\thanks{This work was supported by the National Natural Science Foundation of China through grants 72301008 and 72431001.
S. Zhang and Y. Tang are with the School of Advanced Manufacturing and Robotics, Peking University, Beijing, China. 
        Email: {\tt\small zhangsilan@stu.pku.edu.cn}, {\tt\small yujietang@pku.edu.cn}}}

\date{}

\begin{document}

\maketitle

\begin{abstract}
  This paper investigates zeroth-order (ZO) finite-sum composite optimization. Recently, variance reduction techniques have been applied to ZO methods to mitigate the non-vanishing variance of 2-point estimators in constrained/composite optimization, yielding improved convergence rates. However, existing ZO variance reduction methods typically involve batch sampling of size at least $\Theta(n)$ or $\Theta(d)$, which can be computationally prohibitive for large-scale problems. In this work, we propose a general variance reduction framework, \AlgFullname{} (\AlgAbbrv), which supports flexible implementations---including a pure 2-point zeroth-order algorithm that eliminates the need for large batch sampling. Furthermore, we establish comprehensive convergence guarantees for \AlgAbbrv{} across strongly-convex, convex, and non-convex settings that match their first-order counterparts. Numerical experiments validate the effectiveness of our proposed algorithm.
\end{abstract}

\section{Introduction}
\label{section:introduction}

Zeroth-order (ZO) or derivative-free optimization problems have been the subject of intense studies recently. They have emerged from various practical situations where explicit gradient information of the objective function can be very difficult or even impossible to compute, and only function evaluations are accessible to the decision maker. Zeroth-order optimization methods have seen wide applications in multiple disciplines, including signal processing~\cite{liu2020primer}, parameter tuning for machine learning~\cite{snoek_practical_2012}, black-box adversarial attacks on deep neural networks~\cite{chen_zoo_2017}, model-free control~\cite{ren_lqr_2021}, etc.

In this paper we focus on zeroth-order problems of the following composite finite-sum form,
\begin{equation} \label{eq:intro_formulation}
 \begin{aligned}
    \min_{x \in \mathbb{R}^{d}} & \quad h(x) = f(x) + \psi(x),\\
    \text{with} & \quad f(x) \defeq \frac{1}{n} \sum_{i=1}^{n} f_{i}(x).
\end{aligned}   
\end{equation}
Here each $ f_{i}:\mathbb{R}^d \rightarrow \mathbb{R}, i = 1, 2, \cdots, n$ is differentiable, $\psi : \mathbb{R}^d \rightarrow \mathbb{R} \cup \{+\infty \}$ is convex and lower semicontinuous but possibly non-differentiable. This formulation encompasses convex problems with feasible set constraints and/or non-smooth regularization including 
logistic regression~\cite{ng2004feature}, Cox regression~\cite{10.1093/bioinformatics/btp322} problems with $L_1$ regularization, online sensor detection problems~\cite{chen2018bandit}, etc.
Furthermore, we assume that only function value (zeroth-order) information of each $f_{i}(x)$ is available for the decision maker. 

Various types of zeroth-order algorithms have been proposed to address the challenge of lacking gradient information in the optimization procedure. In this paper, we particularly focus on the line of methods initiated by~\cite{nemirovskij1983problem} and further developed by~\cite{nesterov_random_2017,duchi2015optimal,ghadimi2013stochastic} and other related works. Basically, in this line of methods, a zeroth-order gradient estimator is constructed based on function values acquired at a few randomly explored points, which will then be combined with first-order optimization frameworks, leading to a zeroth-order optimization algorithm. For instance, \cite{nesterov_random_2017} shows that, by combining a two-point Gaussian random gradient estimator with the vanilla gradient descent and Nesterov's accelerated gradient descent iterations, we can obtain a zeroth-order optimization algorithm for deterministic unconstrained smooth convex problems with an oracle complexity of $O(d/\epsilon)$ and $O(d/\sqrt{\epsilon})$, respectively, where $d$ denotes the problem dimension. Similar techniques have also been adapted to the stochastic settings and non-convex settings~\cite{nesterov_random_2017,duchi2015optimal,shamir2017optimal,ghadimi2013stochastic}. Recently, \cite{balasubramanian2022zeroth} proposed a zeroth-order stochastic approximation algorithms based on the Frank-Wolfe method, focusing particularly on handling constraints,
high dimensionality and saddle-point avoiding. \cite{zhang2022new} proposed a one-point feedback scheme that queries the function value once per iteration but achieves comparable performance with two-point zeroth-order algorithms. \cite{kungurtsev2021zeroth} and \cite{pougkakiotis2023zeroth} studied zeroth-order algorithms for stochastic weakly convex optimization. \cite{pmlr-v202-ren23b} investigated how to escape saddle points for zeroth-order algorithms using two-point gradient estimators.

Another challenge of solving~\eqref{eq:intro_formulation} comes from the finite-sum nature of problem~\eqref{eq:intro_formulation}. Even for first-order problems, sampling every $f_{i}(x)$ to do a full-batch gradient descent can be costly, while algorithms like proximal stochastic gradient descent suffer from the variance of the stochastic gradient, resulting in suboptimal convergence rate. Variance reduction techniques greatly improve the rate of first-order stochastic algorithms by reducing the sampling variance. Notably, SVRG~\cite{johnson2013accelerating} algorithm leverages a full-batch gradient snapshot at a past point as control variate, while SAGA~\cite{defazio2014saga} incrementally updates the control variate at each iteration. Algorithms like SARAH~\cite{nguyen2017sarah} and SPIDER~\cite{fang2018spider} use biased gradient estimator to get improved rate under non-convex settings. Katyusha~\cite{allen2018katyusha} and ASVRG~\cite{shang2018asvrg} combine variance reduction and momentum acceleration to achieve even better rates for strongly convex and convex problems, and~\cite{kovalev2020don} uses probabilistic update to construct both loopless version of Katyusha and SVRG.

However, when adopting two-point zeroth-order gradient estimators for solving the problem~\eqref{eq:intro_formulation}, we have to deal with the two-fold variance from both the random sampling of component functions, and the random directions along which we construct the zeroth-order gradient estimator. This problem is especially pronounced in constrained or composite problems. For example, the zeroth-order estimator used in \cite{nesterov_random_2017,duchi2015optimal,shamir2017optimal} has the following form
\begin{equation*}\frac{f_{i}(x+\beta u)-f_{i}(x)}{\beta} u,
\end{equation*}
where $\beta>0$ is a parameter called the \emph{smoothing radius}, and $u$ is a \emph{random perturbation direction} either sampled from the normal distribution $\mathcal{N}(0,I_d)$ or the uniform distribution on the sphere of radius $\sqrt{d}$. As pointed out in~\cite{zhang2024zeroth}, even for composite deterministic problems with only one component function, the random perturbation direction alone would still cause unstable behavior around the optimal point just like stochastic gradient descent, resulting in an inferior rate of $O(1/\epsilon^{2})$. Therefore, simply plugging two-point estimators into existing first-order variance reduction frameworks is not enough to achieve a better convergence rate than stochastic gradient descent. 

Despite these difficulties, some zeroth-order variance reduction algorithms that handles both sources of variance began to emerge. Notably, variants of zeroth-order SVRG and SPIDER algorithms tackles the two-fold variance at the same time by combining $\Theta(nd)$-point and $2$-point zeroth-order gradient estimators, achieving an oracle complexity that in some cases matches their first-order counterparts~\cite{fang2018spider, ji2019improved, rando2025structured}. However, with the exception of \cite{rando2025structured} which deals with non-convex composite problems, these results are restricted to unconstrained problems. Moreover, all these algorithms require $\Theta(nd)$ oracle queries during the full-batch phase, which can be problematic when $n$ and $d$ are large. \cite{huang2019faster} extends both SVRG and SAGA algorithm to zeroth-order non-convex composite problems; however, the proposed algorithms either demand constant $\Theta(d)$-point full-batch sampling (CooSGE), or require a strong assumption of bounded gradient and converge only to neighborhood of the optimal point (GauSGE). \cite{huang2022accelerated} extends the STORM momentum variance reduction scheme~\cite{cutkosky2019momentum} to zeroth-order case. However, their analysis requires a bound on the variance of stochastic zeroth-order gradient independent of problem dimension $d$, which is unreasonable when applied to high-dimensional problems. More recently, \cite{di2024double} combines SVRG and gradient sketching techniques and achieve a rate comparable to first-order algorithms for strongly convex composite problems. However, their algorithm still requires a maximum batch size of $\Theta(n)$; the double variance reduction structure also makes the analysis more involved.

\begin{table}
  \centering
  \caption{Comparison of oracle complexities, average and maximum batch sizes of relevant zeroth-order methods for smooth finite-sum optimization.}
  \label{tb:compare}
  \setlength{\tabcolsep}{3pt}
  \renewcommand{\arraystretch}{1.5}
  \begin{adjustbox}{max width=\textwidth}
    \begin{threeparttable}
     \begin{tabular}{ccccccc} 
    \toprule
    & Strongly convex & Convex & Non-convex & \makecell{Average \\ batch size}& \makecell{Maximum \\ batch size}& Constrained \\[3pt]
    \midrule
       \makecell{Zeroth-order \\ Proximal SGD} & No result &  $\displaystyle O\!\left(\frac{d \sigma^{2}}{\epsilon^{2}}\right) $\,\tnote{1} &  $\displaystyle O\!\left(\frac{d \sigma^{2}}{\epsilon^{2}}\right)$\,\tnote{1} & $\Theta(1)$ & $\Theta(1)$ & Yes\\[8pt]
  \makecell{ZO-SVRG-Coord-Rand\\\cite{ji2019improved}} 
         & No result & $\displaystyle O\!\left(d \max\!\left\{n, \frac{1}{\epsilon}\right\} \ln\frac{1}{\epsilon}\right)$\tnote{2}
         & $\displaystyle O\!\left(d \min\!\left\{\frac{n^{2/3}}{\epsilon}, \frac{1}{\epsilon^{5/3}}\right\}\right)$
         & $\Theta(1)$ 
         & $\Theta(nd)$ 
         & No\\[8pt]
  ZO-SPIDER-Coord~\cite{ji2019improved}  & $\displaystyle O\!\left( d\!\left(\frac{ \sqrt{n}}{\mu} + \frac{m}{\mu^{2}}\right)\ln\frac{1}{\epsilon}\right)$ & No result & $\displaystyle O\!\left(d \min\!\left\{\frac{n^{2/3}}{\epsilon,} \frac{1}{\epsilon^{5/3}}\right\}\right)$ & $\Theta(1)$ & $\Theta(nd)$ & No\\[6pt]
  ZO-ProxSAGA~\cite{huang2019faster} & No result & No result& No result\tnote{3} & $\Theta(1)$ & $\Theta(1)$ & Yes\\
  SPIDER-SZO~\cite{fang2018spider} & No result & No result & $\displaystyle O\!\left(d \min\left\{\frac{n^{1/2}}{\epsilon}, \frac{1}{\epsilon^{3/2}}\right\}\right)$ & $\Theta(1)$ & $\Theta(nd)$ & No\\[3pt]
  ZPDVR~\cite{di2024double} & $ \displaystyle O\!\left(d\!\left(n + \frac{L}{\mu}\right) \ln\left(\frac{1}{\epsilon}\right)\right)$ & No result & No result & $\Theta(1)$ &$\Theta(n)$ & Yes\\
  VR-SZD~\cite{rando2025structured} & No result & No result & $\displaystyle O(n^{2/3}d/\epsilon)$ & $ \displaystyle \Theta(n^{2/3}d)$ & $\Theta(nd)$ & Yes\\
  \midrule
  \AlgAbbrv{} & $\displaystyle O\!\left(d\!\left(n+\frac{L}{\mu}\right)\ln\frac{n}{\epsilon}\right)$
       & $\displaystyle O\!\left(\frac{dn}{\epsilon}\right)$ 
       & $\displaystyle O\!\left(\frac{d^{3/2} n}{\sqrt{R } \epsilon}\right)$\,\tnote{4} 
       & $\Theta\!\left(R \right)$& $\Theta(R )$ & Yes
       \\
   \bottomrule
\end{tabular}   
\begin{tablenotes}
  \item The listed oracle complexities are the number of zeroth-order queries needed to obtain a point $x$ achieving $\mathbb{E}[f(x)-f^\ast]\leq\epsilon$ for the strongly-convex and convex cases, and $\mathbb{E}[\|\nabla f(x)\|_{2}^2]\leq\epsilon$ for the non-convex case, respectively.
  \item[1] Here $\sigma^2$ denotes an upper bound of $\|\nabla f_i(x) - \nabla f(x)\|_{2}^2$ for all $x$ and $i$.
  \item[2] The theoretical analysis imposed boundedness of the iterates as an assumption even in the unconstrained setting.
  \item[3] The work~\cite{huang2019faster} assumed bounded gradient $\|\nabla f_{i}(x)\|_{2}^2 \leq \sigma^{2}$ for all $x$, and only convergence with limited precision $O(d\sigma^2)$ was established for the non-convex case.
  \item[4] $R$ is the batch size of our choice. In the non-convex case, we give complexity results for $R\in [1, n^{2/3}d]$.  The best complexity bound is $O(n^{2/3} d / \epsilon)$ by choosing $R= \lfloor{n^{2/3} d}\rfloor$.
\end{tablenotes}
    \end{threeparttable}
  \end{adjustbox}
\end{table}

\subsection{Main Contributions}
In this work, we propose a general framework, \AlgFullname, to construct variance-reduced zeroth-order gradient estimators for the composite finite-sum problem~\eqref{eq:intro_formulation}. Specifically,
\begin{enumerate}
\item Inspired by the first-order Jacobian sketching scheme~\cite{gower2021stochastic}, the \AlgAbbrv{} framework maintains an incrementally updated Jacobian estimator, and constructs gradient estimator using both historical information stored in the Jacobian estimator and zeroth-order gradient at the current iterate. By adopting different update strategies of the Jacobian estimator, a variety of variance-reduced zeroth-order algorithms can be derived from this framework, covering some existing methods and introducing several new ones. Notably, it admits a pure $2$-point zeroth-order algorithm that does not require any intermittent batch sampling, and a memory-efficient block mini-batch algorithm, as its specific implementations.
\item Moreover, we establish a comprehensive general convergence result for our \AlgAbbrv{} variance reduction framework, for both strongly-convex, convex and non-convex settings. For some specific implementations including the aforementioned $2$-point one, we also provide concrete oracle complexity results that matches first-order algorithms up to an order of $O(d)$.
\end{enumerate}

Table~\ref{tb:compare} compares our \AlgAbbrv{} framework with existing relevant works. Further discussions are as follows: As mentioned before, \cite{huang2019faster} also proposed a $2$-point zeroth-order algorithm similar to ours in convex settings. Compared to~\cite{huang2019faster}, our analysis does not require the assumption of bounded variance and also achieves better result. The convergence result for the non-strongly convex case is also a new one, as most previous works focus on either the strongly convex case or the non-convex case, with the exception of~\cite{ji2019improved}. However, \cite{ji2019improved} only considered the unconstrained case and imposed boundedness of the iterates as an assumption even in the unconstrained setting.

We have also conducted several numerical experiments to demonstrate the effectiveness of our algorithms.

\paragraph{Notations.} For each positive integer $k$, we denote $[k]\defeq\{1,\ldots,k\}$. The $k\times k$ identity matrix will be denoted by $I_k$; we sometimes drop the subscript $k$ if the dimension is clear from the context. The real vector whose entries are all $1$ will be denoted by $\mathbf{1}$. For a finite set $S$, we use $|S|$ to denote its total number of elements. For any real matrices $X,Y$ of the same size, we denote $\langle X,Y\rangle\defeq\tr(X^TY)$; for $x,y\in\mathbb{R}^p$, we simply have $\langle x,y\rangle=x^Ty$. The $\ell_2$ norm on $\mathbb{R}^d$ will be denoted by $\|x\|_2\defeq \sqrt{\langle x,x\rangle}$. For a real matrix $A=[A_{ij}]$, its Frobenius norm is denoted by $\|A\|_F\defeq\sqrt{\langle A,A\rangle}=\big(\sum_{i,j}A_{ij}^2\big)^{1/2}$. The unit sphere in $\mathbb{R}^d$ is denoted by $\mathbb{S}^{d-1}=\setv*{x\in\mathbb{R}^d}{\|x\|_2=1}$.

 \section{Problem Formulation}
In this paper, we focus on solving the following optimization problem 
\begin{equation}
  \label{eq:formulation}
    \min_{x \in \mathbb{R}^{d}}\ h(x), \quad \text{where } h(x) = \frac{1}{n} \sum_{i=1}^{n} f_{i}(x) + \psi(x).
\end{equation}
Here each $ f_{i}:\mathbb{R}^d \rightarrow \mathbb{R}, i = 1, 2, \cdots, n$ is convex and differentiable; $\psi : \mathbb{R}^d \rightarrow \mathbb{R} \cup \{+\infty \}$ is convex and lower semicontinuous but possibly non-differentiable. We shall also denote $f=\frac{1}{n}\sum_{i=1}^n f_i$. We impose the following assumptions on the information that can be accessed by the decision maker with one oracle call:
\begin{itemize}
\item The decision maker is able to efficiently evaluate the following proximal operator
\[
\operatorname{prox}_{\alpha\psi}(x) \defeq
\argmin_{y\in\mathbb{R}^d}\left\{
\frac{1}{2}\|x-y\|_{2}^2 + \alpha\,\psi(y)
\right\}
\]
for any $\alpha>0$.

\item The decision maker may only choose one component function $f_{i}$ and obtain the value of $f_{i}$ by querying a black-box, or a zeroth-order oracle, which returns the value $f_{i}(x)$ whenever a point $x\in\mathbb{R}^d$ is fed into the black-box as its input. The decision maker does not have access to the derivatives of any order for any $f_{i}$.
\end{itemize}
We provide further discussions on the above two assumptions:
\begin{itemize}[leftmargin=0pt, itemindent=\parindent,labelsep=10pt]
\item The first assumption is standard in composite optimization and has been imposed in relevant studies such as~\cite{huang2019faster,pougkakiotis2023zeroth}. It allows us to incorporate simple non-smoothness into our problem settings and algorithm design. Two representative examples are given below.
\begin{enumerate}
    \item If we take $\psi$ to be the indicator function of a closed and convex set $\mathcal{X}\subseteq\mathbb{R}^d$:
\[
\psi(x) = I_{\mathcal{X}}(x) = \begin{cases}
0, & \text{if}\ x\in\mathcal{X}, \\
+\infty, & \text{otherwise},
\end{cases}
\]
Then the main problem~\eqref{eq:formulation} is equivalent to the constrained optimization problem
$
\min_{x\in\mathcal{X}} f(x)
$,
and the proximal operator reduces to the projection operator onto the convex set $\mathcal{X}$. We see that our problem formulation includes set-constrained optimization as a special case.
\item Another practical example is the lasso regression problem~\cite{tibshirani1996lasso}, in which each $f_{i}(x)$ would be the loss function of linear or logistic regression corresponding to individual data points, while $\psi(x) = \lambda \lVert x\rVert_{1}$ for some $\lambda > 0$ is a non-smooth regularization term that enforces sparsity on the solution $x$.
\end{enumerate}

\item The second assumption, as we have shown in Section~\ref{section:introduction}, is relevant to many problems in areas such as signal processing and machine learning, and necessitates the development of a zeroth-order optimization algorithm. To overcome the difficulty of lacking gradient information, different kinds of zeroth-order gradient estimators have been proposed as substitutes of first-order gradient. Below we highlight two common choices:
    \begin{enumerate}
        \item For each $f_i$, we can construct the 2-point estimator with random perturbation direction $u \in \mathbb{R}^d$ and smoothing radius $\beta > 0$:
        \begin{equation}\label{eq:fm_2p}
            \widehat \nabla_u f_i(x; \beta) \defeq \frac{f_i(x+ \beta u) - f_i(x)}{\beta}u.
        \end{equation}
        With a unit vector $u$, it is shown in Lemma \ref{lm:direc_bias} that $\widehat \nabla_{u} f_{i}(x ; \beta) \approx u^{T} u f_{i}(x)$, i.e., $\widehat \nabla_{u} f_{i}(x ; \beta)$ effectively estimates the directional derivative. Moreover, it has been established in~\cite{nesterov_random_2017} that, when $u$ is sampled from $\mathcal{N}(0, I_d)$ and $\beta$ is small enough, \eqref{eq:fm_2p} provides an approximately unbiased estimator of the gradient $\nabla f_i(x)$, while its variance is roughly proportional to the squared magnitude of the gradient:
        \begin{equation}\label{eq:2_point_bias_var}
        \mathbb{E}\!\left[\widehat\nabla_u f_i(x; \beta)\right] = \nabla f_i(x) + O(\beta), \quad \operatorname{Var}\!\left(\widehat\nabla_u f_i(x; \beta)\right) = \Theta\!\left(d\|\nabla f_i(x)\|_{2}^2\right) + O(\beta^2).
        \end{equation}
        For finite-sum composite problems, this implies that the variance of 2-point estimators will in general not vanish near the optimal point, forcing us to employ diminishing step sizes that will in turn slow down the convergence rate.
        
    \item Another method to construct zeroth-order estimators is to sample along every coordinate direction, resulting in the following $(d+1)$-point estimator:
        \begin{equation}\label{eq:fm_dp}
            \widehat \nabla f_i(x; \beta) \defeq \sum_{j=1}^{d} \frac{f_i(x+ \beta e_j) - f_i(x)}{\beta}e_j,
        \end{equation}
        where $e_j$ is the $j$-th standard basis vector in $\mathbb{R}^d$. This deterministic estimator has no variance and produces fairly accurate approximation of $\nabla f_i(x)$ when $f_i$ is smooth and $\beta$ is sufficiently small. However, its computation will require $d+1$ samples, which can be problematic when $d$ is large.
    \end{enumerate}

In this work, we shall propose a general framework to construct variance-reduced zeroth-order gradient estimators, to address the aforementioned issues.
\end{itemize}

We also make the following technical assumption that will facilitate convergence analysis of our algorithm:
\begin{assumption}
\label{as:smooth}
Each function $f_i$ is $L$-smooth on $\mathbb{R}^d$ for some $L>0$, i.e.,
\[
\|\nabla f_{i}(x)-\nabla f_{i}(y)\|_{2}\leq L\|x-y\|_{2},\qquad\forall x,y\in\mathbb{R}^d.
\]
\end{assumption}

Note that Assumption~\ref{as:smooth} implies $f=\frac{1}{n}\sum_{i=1}^n f_i$ is $L$-smooth.

 \section{Algorithm}
\label{section:algorithm}

In this section, we present our main algorithm, \emph{\AlgFullname{}} (\AlgAbbrv{}). This is a general algorithmic framework that allows various implementations.

As is mentioned in the previous section, if at each iteration we randomly sample a component function $f_i$ and a perturbation direction $u$ and plug the 2-point estimator $\widehat \nabla_u f_i(x; \beta)$ into the proximal gradient descent algorithm, then the variance of the estimator will not vanish around the optimal point, resulting in an inferior rate of $O(d \sigma^2/\epsilon^2)$, where $\sigma^2$ denotes an upper bound of $\|\nabla f_i(x) - \nabla f(x)\|_{2}^2$ for all $x$ and $i$. 
If we instead use the deterministic $(d+1)$-point estimator~\eqref{eq:fm_dp} to construct a fairly accurate estimate of $\nabla f(x)$, then since $f$ is the average of $n$ component functions for our finite-sum problem~\eqref{eq:formulation}, we will need $O(nd)$ zeroth-order samples per iteration, which makes the algorithm very inflexible for large scale problems.

To overcome the disadvantages of the above two choices of zeroth-order estimators, we aim to design a stochastic estimator $g_k$ whose variance vanishes as the iterate $x_k$ approaches an optimal point. This can be achieved by adopting variance reduction techniques.
We now present the details of the design of our variance-reduced gradient estimator $g_k$.

\paragraph{Jacobian Estimator $J_k$}
For each iteration $k$, our algorithm keeps a tracking Jacobian estimator $J_{k}$ that groups the gradient information of all component functions in a $d \times n$ matrix, its $n$ columns being estimates of the gradients of the $n$ component functions $\nabla f_{i},\,i \in[n]$. For notational simplicity, we define the vector-valued function $F: \mathbb{R}^{d} \rightarrow \mathbb{R}^{1\times n}$ by arranging the $n$ component functions into a row vector:
\begin{equation*}
  F(x) \defeq \begin{bmatrix} f_{1}(x) & f_2(x) & \cdots & f_{n}(x) \end{bmatrix}\in \mathbb{R}^{1\times n},
\end{equation*}
and denote its Jacobian at $x$ by $\nabla F(x)=[\partial f_j/\partial x_i]\in\mathbb{R}^{d\times n}$. Then we hope to achieve $J_{k}\approx \nabla F(x_{k})$ by updating $J_{k}$ using zeroth-order gradient estimators.

However, since by each 2-point estimator we can only obtain approximate directional gradient information of one component function along one direction, fully updating $J_{k}$ requires computing zeroth-order gradient estimators along $d$ coordinate directions for all $n$ component functions, resulting in $O(nd)$ oracle calls per iteration, which can be costly when $nd$ is large. So instead, at each iteration $k$, our algorithm randomly selects one part of $J_{k}$ to update, and leave the remaining parts unchanged.

Specifically, at each iteration $k$, our algorithm randomly generates a set $\mathcal{S}_k \subseteq [n]\times \mathbb{R}^d$. Each element of $\mathcal{S}_k$ is an ordered pair of the form $(i, u)$ and corresponds to one zeroth-order estimator $\widehat \nabla_u f_i(x_k; \beta_k)$ defined by~\eqref{eq:fm_2p}, with $i \in [n]$ specifying the index of an component function $f_i$, and $u$ being the perturbation direction. Then, we only update the part of $J_k$ indicated by these pairs in $\mathcal{S}_k$:
\label{eq:alg_J_update}
    \begin{equation}\label{eq:alg_J_update_1}
        J_{k+1}^{(i)} = 
        \begin{cases}
          J_{k}^{(i)} + \sum_{u:(i, u) \in \mathcal{S}_k} \!\left(\widehat \nabla_{u} f_{i} (x_{k};\beta_k) - u u^{T} J_{k}^{(i)}\right), & \text{if }{i} \in \mathcal{I}_{k}, \\
          J_{k}^{(i)}, & \text{otherwise},
        \end{cases}
    \end{equation}
where $J_k^{(i)}$ denotes the $i$-th column of $J_k$,  and $\mathcal{I}_k \defeq \setv*{i}{\exists u\text{ s.t. }(i, u)\in \mathcal{S}_k}$ collects the indices that have appeared as the first entries of order pairs in $\mathcal{S}_k$. In other words, for every $(i, u) \in \mathcal{S}_k$, we set the projection of $J_{k+1}^{(i)}$ along direction $u$ to be $\widehat \nabla_u f_i(x_k;\beta_k)$, which approximates the true directional gradient $\langle\nabla f_i(x_k),u\rangle u=uu^T \nabla f_i(x_k)$.

Intuitively, supposing that $x_k$ has approximately converged, with each update~\eqref{eq:alg_J_update_1}, $J_k$ will get closer to the true Jacobian as long as the choice of $\mathcal{S}_k$ ensures a sufficient proportion of $J_k$ to be updated on average; the precise conditions on $\mathcal{S}_k$ will be stated in Assumption~\ref{as:algorithm_rand_var} later.

We further introduce the following linear operator that denotes the combined left and right projection:
    \begin{equation}
      \label{eq:alg_p_def}
        P_k(A) \defeq \sum_{(i, u)\in \mathcal{S}_k} uu^T A e_i e_i^T, \quad \forall A\in \mathbb{R}^{d\times n}.
    \end{equation}
    so that~\eqref{eq:alg_J_update_1} can be more succinctly reformulated as
    \begin{equation}\label{eq:alg_J_update_2}
        J_{k+1} =  J_{k} +
  G_{\beta_k}(x_{k}; \mathcal{S}_{k})
  - P_k(J_k),
    \end{equation}
    where
    \begin{equation}
    \label{eq:biG_def}
          G_{\beta_k}(x_k; \mathcal{S}_{k}) =  
  \sum_{(i, u) \in \mathcal{S}_{k}}
  \widehat \nabla_{u} f_{i}(x_{k} ; \beta_{k}) e_{i}^{T}
    \end{equation}
    is a $d \times n$ matrix. As will be shown later by Lemma~\ref{lm:projected_bias}, we have $ G_{\beta_k}(x_k; \mathcal{S}_{k}) \approx P_k (\nabla F(x_k))$.
    
    To incorporate further flexibility into our algorithm framework, we shall also introduce a Bernoulli random variable $\omega_k$ that ``masks'' the update of $J_k$, so that when $\omega_k = 0$, $J_k$ is not updated; when $\omega_k = 1$, $J_k$ is updated according to~\eqref{eq:alg_J_update}. In other words, we modify~\eqref{eq:alg_J_update} as
    \begin{subequations}
    \label{eq:alg_J_update_final}
    \begin{equation}
  J_{k+1}^{(i)} = 
        \begin{cases}
          J_{k}^{(i)} + \sum_{u:(i, u) \in \mathcal{S}_k} \!\left(\widehat \nabla_{u} f_{i} (x_{k};\beta_k) - u u^{T} J_{k}^{(i)}\right), & \text{if }\omega_k=1\text{ and }i\in\mathcal{I}_k,\\
          J_{k}^{(i)}, & \text{otherwise}.
        \end{cases}
    \end{equation}
    or, written more compactly,
\begin{equation}
J_{k+1} = J_k+\omega_k(G_{\beta_k}(x_{k};  \mathcal{S}_{k})
  - P_k(J_{k})).
\end{equation}
\end{subequations}

We now state the conditions on the random quantities $\mathcal{S}_{k}$ and $\omega_{k}$.
    \begin{assumption}\label{as:algorithm_rand_var}
      The following conditions hold:
      \begin{itemize}
        \item $(\mathcal{S}_k,\omega_k)$ for $k=0,1,2,\ldots$ are i.i.d.
         \item For any step $k$ and component index $i \in [n]$, the set of sampled directions associated with $i$, defined as $\mathcal{U}_{k, i} \defeq \setv*{u}{(i, u) \in \mathcal{S}_k }$, consists of orthonormal vectors. Specifically, for any $u \in \mathcal{U}_{k, i}$ we have $\|u\|_2 = 1$, and for any distinct $u, v \in \mathcal{U}_{k, i}$ we have $u^\top v = 0$.
          \item There exists a constant $\sigma \in(0,1)$ such that
        \begin{equation}
        \label{eq:alg_omega_req1}
          \mathbb{E}[\omega_{k} P_{k}(A)] = \sigma A
        \end{equation}
        for any $A\in\mathbb{R}^{d\times n}$ for all $k$.
      \end{itemize}
    
    \end{assumption}
    The second requirement on orthogonality is introduced to facilitate a tighter error analysis by eliminating cross-terms in the estimation variance. The third requirement \eqref{eq:alg_omega_req1} states that at each iteration, a constant fraction $\sigma $ of $J_k$ will be updated on average. Note that we allow $\mathcal{S}_k$ to be dependent on $\omega_k$.
  
\begin{remark}
  \label{rm:alg_sampling}
  To satisfy the orthonormality requirement in Assumption~\ref{as:algorithm_rand_var}, we propose the following general approach for constructing the set $\mathcal{S}_k$. At each iteration, first generate an orthogonal matrix $Q_k \in \mathbb{R}^{d \times d}$. Then, sample $|\mathcal{S}_k|$ \emph{distinct} pairs $(i, j)$ from the set $[n] \times [d]$. Finally, populate $\mathcal{S}_k$ with elements of the form $(i, Q_k^{(j)})$, where $Q_k^{(j)}$ denotes the $j$-th column of $Q_k$. Since the columns of $Q_k$ are mutually orthogonal and the indices $j$ associated with any single component $i$ are distinct, the requirement for $\mathcal{U}_{k,i}$ to consist of orthonormal vectors is satisfied. Regarding the generation of the orthogonal matrix $Q_k$, we propose two primary schemes:
  \begin{enumerate}
    \item \textbf{Coordinate Directions:} By setting $Q_k = I$ for all $k \geq 0$, the Jacobian estimator $J_k$ is always updated along standard coordinate directions. This approach induces no additional cost as it avoids matrix decompositions.
    
    \item \textbf{Spherical Random Directions:} We can generate a random orthogonal matrix $Q_k$ at each iteration. This is achieved by generating a random matrix $Z_k \in \mathbb{R}^{d \times d}$ with i.i.d. standard Gaussian entries and performing a QR decomposition $Z_k = Q_k R$. While this requires more computation than the coordinate scheme, the resulting columns $Q_k^{(j)}$ are uniformly distributed over the unit sphere $\mathbb{S}^{d-1}$. This allows for a potentially tighter bound on the bias of the zeroth-order estimator by leveraging the zeroth-order smoothing technique in~\cite{nesterov_random_2017}.
  \end{enumerate}
\end{remark}

\paragraph{Gradient Estimator $g_k$}
Now we can use $J_k$ as a reference to construct a variance-reduced gradient estimator $g_k$. We start by sampling another set $\mathcal{R}_{k} \subseteq [n]\times \mathbb{R}^{d}$, with fixed size $\lvert \mathcal{R}_{k}\rvert = R, \forall k \geq 1$.
To construct the $R$ elements of $\mathcal{R}_{k}$, we first sample $R$ distinct indices 
$i_{k}^{1}, i_{k}^{2}, \ldots, i_{k}^{R} \in [n]$ uniformly without replacement. 
We then sample $R$ independent random directions 
$u_{k}^{1}, u_{k}^{2}, \ldots, u_{k}^{R} \in \mathbb{R}^{d}$, each drawn from either the uniform distribution over the standard basis $\{e_1,\ldots,e_d\}$, or the uniform distribution over the unit sphere $\mathbb{S}^{d-1}$. We then let
\[
\mathcal{R}_{k}= \left\{(i_{k}^{1}, u_{k}^{1}),\, (i_{k}^{2}, u_{k}^{2}),\, \ldots,\, (i_{k}^{R}, u_{k}^{R})\right\}.
\]
Although both $\mathcal{R}_{k}$ and $\mathcal{S}_{k}$ consist of index--direction pairs, the 
sampling requirements for $\mathcal{R}_{k}$ are more restrictive, as its elements must be drawn 
i.i.d.\ from specific prescribed distributions.

  Given $\mathcal{R}_{k}$, our algorithm employs the following construction of the gradient estimator $g_k$:
    \begin{equation}
    \label{eq:alg_gk_def}
g_{k} 
  = \frac{1}{n}J_k\cdot\mathbf{1} + \frac{d}{R}\sum_{(i, u) \in\mathcal{R}_k}
  \left(
  \widehat\nabla_u f_i(x_k;\beta_k)
  -uu^T J_k^{(i)}
  \right).
    \end{equation}
    By our sampling strategy of the random indices and directions in $\mathcal{R}_k$, we have
    \begin{equation*}
      \mathbb{E}\Econd*{d u u^{T} J_{k}^{(i)}}{x_{k}, J_{k}} = \frac{1}{n} J_{k} \cdot \mathbf{1},
    \end{equation*}
    and because the elements in $\mathcal{R}_{k}$ are i.i.d., we can conclude that 
    \begin{equation*}
      \mathbb{E}\Econd*{g_{k}}{x_{k}, J_{k}} 
      = \mathbb{E}\Econd*{\widehat \nabla_{u} f_{i}(x_{k} ; \beta_{k})}{x_{k}, J_{k}} 
      \approx \nabla f(x), 
    \end{equation*}
    where the last approximation follows from~\eqref{eq:2_point_bias_var}.
    In other words, $g_{k}$ serves as an approximately unbiased stochastic estimator of $f$ at $x_{k}$.
    Moreover, given the approximation $\widehat \nabla_{u} f_{i}(x_{k} ; \beta_{k}) \approx  u u^{T} f_{i}(x_{k}) $, we further derive that
    \begin{align*}
    \operatorname{Var}\Vcond*{g_k}{x_k,J_k} &= \operatorname{Var}\Vcond*{
      \frac{d}{R} \sum_{(i, u) \in\mathcal{R}_k}
  \left( \widehat\nabla_u f_i(x_k;\beta_k) -uu^T J_k^{(i)} \right)
    }{x_k,J_k}\\
  &\approx \frac{d}{R} \operatorname{Var}\Vcond*{ u u^{T} (\nabla f_{i}(x^{k}) - J_{k}^{(i)})}{x_k,J_k}
  \\&= \frac{d}{R} \operatorname{Var}\Vcond*{ u u^{T} (\nabla F(x^{k}) - J_{k}) e_{i} e_{i}^{T}}{x_k,J_k}.
    \end{align*}
indicating that the variance of $g_k$ will be small as long as $J_k$ tracks the real Jacobian matrix $\nabla F(x_k)$ relatively accurately. The update rule for the iterate $x_k$ is then given by

\begin{equation}
\label{eq:alg_prox_grad}
x_{k+1} = \operatorname{prox}_{\alpha\psi}(x_k-\alpha g_k),
\end{equation}
where $\alpha>0$ is the step size.

Wrapping up all the ingredients, we complete the design of our proposed algorithm, \AlgFullname{}, which is summarized in Algorithm~\ref{alg:zips}.

\begin{algorithm}[t]
\DontPrintSemicolon
\SetAlgoLined
\LinesNumbered
\SetKwInput{KwParam}{Parameters}
\SetKwInput{KwInit}{Initialization}
\KwParam{Stepsize $\alpha > 0$; smoothing radius $\beta_k>0$ for $k\geq 0$, batch size $R \in [nd]$.}
\KwInit{ $x_{0}\in\mathbb{R}^d$; $J_{0} \in\mathbb{R}^{d\times n}$.}
\For{$k=0,1,2,\ldots$}{
  Generate a Bernoulli random variable $\omega_k$.\;
  Generate $\mathcal{S}_k\subseteq[n]\times\mathbb{R}^d$ that satisfies Assumption~\ref{as:algorithm_rand_var}.\;
  Sample
  $i_{k}^{1}, \ldots, i_{k}^{R}$ uniformly from $[n]$ without replacement. \;
  Sample $u_{k}^{1}, \ldots, u_{k}^{R}$ independently and uniformly from $\{e_1,\ldots,e_d\}$ or $\mathbb{S}^{d-1}$.\;
$
\mathcal{R}_{k}= \left\{(i_{k}^{1}, u_{k}^{1}),\, (i_{k}^{2}, u_{k}^{2}),\, \ldots,\, (i_{k}^{R}, u_{k}^{R})\right\}$. \;
  \vspace{3pt}
  Construct the gradient estimator $g_k$ by
  \[
  \displaystyle g_{k} 
  = \frac{1}{n}J_k\cdot\mathbf{1} + \frac{d}{R}\sum_{(i, u) \in\mathcal{R}_k}
  \left(\widehat\nabla_u f_i(x_k;\beta_k)
  -uu^T J_k^{(i)}
  \right).
  \]\;\vspace{-12pt}
  Update $x_{k+1}$ by
  $
  \displaystyle x_{k+1} = \operatorname{prox}_{\alpha \psi} (x_{k} - \alpha g_{k})$.\;\vspace{3pt}
  For each $i\in[n]$, update the $i$-th column of $J_{k+1}$ by
       \[
  J_{k+1}^{(i)} = 
        \begin{cases}
          J_{k}^{(i)} + \sum_{u:(i, u) \in \mathcal{S}_k} \!\left(\widehat \nabla_{u} f_{i} (x_{k};\beta_k) - u u^{T} J_{k}^{(i)}\right), & \text{if }\omega_k=1\text{ and }{i} \in \mathcal{I}_{k},\\
          J_{k}^{(i)}, & \text{otherwise}.
        \end{cases}\]
    }
\caption{\AlgFullname{}}\label{alg:zips}
\end{algorithm}

\subsection{Examples of Implementations}
\label{subsection:impl}

Our proposed \AlgAbbrv{} framework is quite general and provides much flexibility in the choice of the probability distributions of $\mathcal{S}_k$ and $\omega_k$ . Here we give some examples of specific implementations.

    \paragraph{Implementation I}
    We set $\omega_{k} = 1$ and fix $| \mathcal{S}_k | = R$ for all $k$, which implies that a partial update of $J_k$ occurs at every iteration. To construct $\mathcal{S}_k$, we follow the general procedure described in Remark~\ref{rm:alg_sampling}: we first sample $R$ distinct pairs $(i, j)$ uniformly from $[n] \times [d]$, and then map these indices to direction vectors using an orthogonal matrix $Q_k$. By employing either the Coordinate Directions or Spherical Random Directions scheme from Remark~\ref{rm:alg_sampling}, the condition \eqref{eq:alg_omega_req1} is automatically satisfied with $\sigma = \frac{R}{nd}$. Furthermore, this implementation requires only $O(R)$ function evaluations per iteration since $\lvert \mathcal{S}_k\rvert = \lvert \mathcal{R}_k\rvert = R$.

    This strategy of generating $(\mathcal{R}_k,\mathcal{S}_k,\omega_k)$ can be viewed as a combination of SAGA~\cite{defazio2014saga} and SEGA~\cite{hazely2018sega}. Unlike existing algorithms such as ZO-SVRG~\cite{ji2019improved}, SPIDER-SZO~\cite{fang2018spider} or ZO-DVR~\cite{di2024double}, here the update of $J_{k}$ is purely incremental and no full-batch sampling (i.e., $R=nd$) is necessary. The mini-batch size $R$ can be as small as $1$, in which case only two oracle calls are needed at each iteration, making the algorithm favorable for high-dimensional problems.

    \paragraph{Implementation II} We let $\omega_k \sim \mathrm{Bernoulli}(p)$, with $p = \min\{R/d, 1\}$. If $R < d$, the update of $J_k$ is performed with probability $R/d$; otherwise, the update happens at each step. When $\omega_k = 1$, we sample a subset of component indices $\mathcal{I}_k$ from $[n]$ uniformly at random, with size $\lvert \mathcal{I}_k \rvert = \lceil R/pd \rceil$ (note that $\lvert\mathcal{I}_k\rvert=1$ if $R<d$). We then construct the set $\mathcal{S}_k$ using the columns of the orthogonal matrix $Q_k$ (as defined in Remark~\ref{rm:alg_sampling}):
\[
    \mathcal{S}_k = \mathcal{I}_k \times \left\{Q_k^{(1)}, Q_k^{(2)}, \ldots, Q_k^{(d)}\right\}.
\]
This means we select a subset of columns of $J_k$ and update them along all $d$ orthogonal directions provided by $Q_k$. In this setting, if $R < d$, \eqref{eq:alg_omega_req1} is satisfied with $\sigma = \frac{R}{nd}$; if, however, $R \geq d$, then we have $\sigma = \frac{\lceil R/d \rceil}{n}$. More specifically, when $\omega_{k} = 0$, we have $J_{k+1} = J_{k}$. When $\omega_{k} = 1$, we use $O(d)$-point full batch zeroth-order estimators to fully update the selected columns of $J_k$:
\[
   J_k^{(i)} = \widehat \nabla f_i(x_{k}; \beta_k), \quad \forall i \in \mathcal{I}_k.
\]
This update strategy resembles the SAGA~\cite{defazio2014saga} algorithm, employing conditional updates or mini-batching depending on $p$. A simple computation shows that the average number of oracle calls per iteration is $O(R)$, matching Implementation I.

\paragraph{Implementation III} We let $\omega_k \sim \mathrm{Bernoulli}({R}/{nd})$ and set $\mathcal{S}_k$ to include all pairs of component indices and orthogonal directions:
\[
    \mathcal{S}_k = [n] \times \{Q_k^{(1)}, Q_k^{(2)}, \ldots, Q_k^{(d)}\}.
\]
In this case, the update of $J_k$ is performed with probability $\frac{R}{nd}$. When the update occurs, we compute an $O(d)$-point zeroth-order gradient estimator for every component function, resulting in a complete update of the Jacobian matrix $J_k$:
\[
J_k^{(i)} = \widehat \nabla f_i(x_{k}; \beta_k), \quad \forall i \in [n].
\]
Therefore \eqref{eq:alg_omega_req1} is satisfied with $\sigma = \frac{R}{nd}$. This scheme is similar to the ZO-SVRG~\cite{ji2019improved} algorithm. However, our approach uses a probabilistic update for $J_{k}$ rather than a double-loop structure, rendering the algorithm more flexible and easier to analyze. Although a full update of $J_k$ requires $O(nd)$ oracle calls, it occurs infrequently, ensuring that the average number of oracle calls per iteration remains $O(R)$.

    \subsection{A Memory-Efficient Implementation}
    \label{subsection:mem_eff}
     A straightforward implementation of Algorithm~\ref{alg:zips} requires storing the matrix $J_{k}$ that contains $nd$ entries, which can be unfavorable in situations where memory is a bottleneck. In this subsection, we present another implementation of Algorithm~\ref{alg:zips} that balances memory complexity and maximum batch size.

     In this implementation, we split the $d\times n$ matrix $J_k$ into $B$ blocks, denoted by disjoint sets of indices $\mathcal{T}_{1}, \mathcal{T}_{2}, \ldots, \mathcal{T}_{B}$ satisfying $\bigcup_{l=1}^{B} \mathcal{T}_{l} = [n]\times [d]$, so that $(i, j) \in \mathcal{T}_{l}$ means the $(i, j)$-th entry falls in the $l$-th block. Here we will not explicitly store $J_{k}$, but will instead store $\tilde g_{k} = \frac{1}{n} J_{k} \mathbf{1}$ and a snapshot point $\tilde x^{l}_{k}$ for each block $\mathcal{T}_{l}$, which will be updated probabilistically. Note that since we store a $d$-dimensional vector $\tilde x^{l}_{k}$ for each block, to actually save memory, we require that $B < n$. The update process is again controlled by the Bernoulli variable $\omega_{k}$:
     \begin{itemize}
         \item If $\omega_{k} = 0$, no update to $\tilde g_{k}$ and any $\tilde x^{l}_{k}$ would occur, and $g_{k}$ is computed by
     \begin{equation}
       \label{eq:alg_mem_eff_g}
  g_{k} 
  = \tilde g_{k}
  + \frac{d}{R}\sum_{(i, e_j) \in\mathcal{R}_k}
  \left(\widehat\nabla_{e_{j}} f_i(x_k;\beta_k)
  - \sum_{l=1}^{B}\mathbb{1}_{\mathcal{T}_{l}}(i,j)\cdot \widehat \nabla_{e_{j}} f_{i}(\tilde x^{l}_{k};\beta_k)
  \right).
     \end{equation}
     Here the members of the random set $\mathcal{R}_k=\left\{(i_{k}^{1},u_{k}^{1}),\ldots,(i_{k}^{R},u_{k}^{R})\right\}$ are generated as follows: (i) Sample $R$ distinct indices 
$i_{k}^{1}, \ldots, i_{k}^{R} \in [n]$ uniformly without replacement; (ii) sample $R$ independent vectors $u_{k}^{1}, \ldots, u_{k}^{R}$ uniformly from the standard basis $\{e_1,\ldots,e_d\}$. Note that as opposed to the previous implementations, we require the directions in $\mathcal{R}_{k}$ to be coordinate directions. We use $\mathbb{1}_{\mathcal{T}_{l}}(i,j)$ to denote the function that equals $1$ if $(i,j)\in\mathcal{T}_l$ and equals $0$ otherwise. 
     
     \item If $\omega_{k} = 1$, we randomly pick $l_k\in[B]$, update $\tilde x^{l_k}_{k+1}$ to be the current iterate $x^{k}$, and update $\tilde g_{k+1}$ by
     \begin{equation*}
       \tilde g_{k+1} = \tilde g_{k} + \frac{1}{\lvert \mathcal{T}_{l_k}\rvert} \sum_{(i, j) \in \mathcal{T}_{l_k}}\!\!\big(\widehat \nabla_{e_{j}} f_{i}(x_{k};\beta_k) - \widehat \nabla_{e_{j}} f_{i}(\tilde x^{l_k}_{k};\beta_k)\big),
     \end{equation*}
     with $g_{k}$ being given by \eqref{eq:alg_mem_eff_g}
     \end{itemize} 
     The detailed steps are also given in Algorithm~\ref{alg:zips_mem_eff}.
     
  \begin{algorithm}[t]
    \DontPrintSemicolon
    \SetAlgoLined
    \LinesNumbered
    \SetKwInput{KwParam}{Parameters}
    \SetKwInput{KwInit}{Initialization}
    \KwParam{Stepsize $\alpha > 0$; smoothing radius $\beta_k>0$ for each $k\geq 0$; updating probability $ 0 < p < 1$; batch size $R\in [nd], R\leq nd/B$.}
    \KwInit{ $x_{0}\in\mathbb{R}^d$ ; disjoint sets of indices $\mathcal{T}_{1}, \mathcal{T}_{2}, \ldots, \mathcal{T}_{B}$; $\tilde x^{1}_{0} = \tilde x^{2}_{0} = \cdots = \tilde x^{B}_{0} = x_{0}$; $\tilde g_{0} \in\mathbb{R}^d$.}
    \For{$k=0,1,2,\ldots$}{
      Generate $\omega_{k} \sim \operatorname{Bernoulli}(p)$ and the random set $\mathcal{R}_k$ as described in Section~\ref{subsection:mem_eff}.\;
      \vspace{3pt}
      $\displaystyle g_{k} 
  = \tilde g_{k}
  + \frac{d}{R}\sum_{(i, e_{j}) \in\mathcal{R}_k}
  \left(\widehat\nabla_{e_{j}} f_i(x_k;\beta_k)
  - \sum_{l=1}^{B}\mathbb{1}_{\mathcal{T}_{l}}(i,j)\cdot \widehat \nabla_{e_{j}} f_{i}(\tilde x^{l}_{k};\beta_k)
  \right) $ \;
  \vspace{3pt}
      \uIf{$\omega_k = 1$}{
      \vspace{3pt}
      Pick a random index $l_k \in [B]$ uniformly at random.\;
      \vspace{3pt}
    $\displaystyle \tilde g_{k+1} = \tilde g_{k} + \frac{1}{\lvert \mathcal{T}_{l_k}\rvert} \sum_{(i, j) \in \mathcal{T}_{l_k}}
    \big(\widehat \nabla_{e_{j}} f_{i}(x_{k};\beta_k) - \widehat \nabla_{e_{j}} f_{i}(\tilde x^{l_k}_k;\beta_k)\big)$.\;
      \vspace{3pt}
      $\displaystyle\tilde x_{k+1}^{m} = 
      \begin{cases}
      x_{k}, & \text{for }m=l_k, \\
      \tilde x_{k}^{m}, & \text{for }m\in[B]\backslash\{l_k\}.
             \end{cases}$\;
      }
      \Else{
      $\tilde x_{k+1}^{m} = \tilde x_{k}^{m}, \ m=1, \ldots,B$.\;
      \vspace{3pt}
      $\tilde g_{k+1} = \tilde g_{k}$.
      }
      \vspace{3pt}
      $\displaystyle x_{k+1} = \text{prox}_{\alpha \psi} (x_{k} - \alpha g_{k})$.\;
    }
    \caption{Memory-Efficient \AlgAbbrv{}}\label{alg:zips_mem_eff}
  \end{algorithm}

  As a special implementation of Algorithm~\ref{alg:zips}, Algorithm~\ref{alg:zips_mem_eff} only needs to store $B$ vectors $\tilde x^{1}_{k},\tilde x^{2}_{k}, \ldots, \tilde x^{B}_{k}$ plus $\tilde g_{k}$, the running average for their gradients, which results in a memory complexity of $\Theta(Bd)$. Furthermore, when the partitions $\mathcal{T}_{1}, \mathcal{T}_{2}, \ldots, \mathcal{T}_{B}$ are roughly of the same size, the maximum batch size is $\Theta(nd/B)$, so by adjusting the partition number $B$ properly, one can achieve a balanced trade-off between memory complexity and maximum batch size. This makes the algorithm more flexible under some scenarios where memory can be a bottleneck.

 \section{Main Results}
\newcounter{betacounter}
\setcounter{betacounter}{0}
\label{section:result}
In this section, we present convergence results of Algorithm~\ref{alg:zips} under different settings and parameter regimes. The proofs of these results will be postponed to Section~\ref{section:proof}.
Before presenting our main results, we define the following quantity to make our proof cleaner:
\begin{equation*}
  \nu \defeq \mathbb{E}\!\left[ \omega_{k} \lvert \mathcal{S}_{k}\rvert\right]
\end{equation*}
Note that $\nu$ is independent of the choice of $k$ since $( \mathcal{S}_k, \omega_k), k = 1, 2, \ldots$ are i.i.d.\ under Assumption~\ref{as:algorithm_rand_var}.

As for the value of $\sigma$ and $\nu$ for the implementations mentioned before, we provide the following lemma:
\begin{lemma}\label{lm:param}
Consider any one of the three implementations in Section~\ref{subsection:impl} or the memory-efficient implementation in Section~\ref{subsection:mem_eff} with $p=BR/(nd)$, we have $R/(nd)\leq\sigma \leq 2R/(nd)$ and $\nu=nd\sigma$.
\end{lemma}

\subsection{Strongly Convex Case}
We first consider the situation where each $f_i$ is convex and $f=\frac{1}{n} \sum_{i=1}^{n} f_{i}$ is $\mu$-strongly convex for some $\mu>0$.

The convergence results of \AlgAbbrv{} for the strongly convex case are summarized in the following theorem.
\begin{theorem}
  \label{tm:sc_main}
Suppose Assumptions~\ref{as:smooth} and \ref{as:algorithm_rand_var} hold, each $f_i$ is convex, and $f=\frac{1}{n}\sum_{i=1}^n f_i$ is $\mu$-strongly convex for some $\mu>0$. Let $x^\ast$ denote the optimal solution to~\eqref{eq:formulation}.
By setting the step size to be $\alpha = \frac{R}{2 L (36 d+R)}$, we have
  \refstepcounter{betacounter}
  \label{counter:sc_final}
  \begin{equation}
    \label{eq:sc_main}
    \mathbb{E}\!\left[\lVert x_{K} - x^{\ast}\rVert_{2}^{2}\right] 
    \leq 
    (1\!-\!\kappa)^{K}\!
    \left[\lVert x_{0} \!-\! x^{\ast}\rVert _{2}^{2}
    + \frac{8dR}{L^2n\sigma  (36 d\!+\!R)^2}
    \|J_0 \!-\! \nabla F(x^\ast)\|_F^2
    \right]
    + \sum_{k=1}^{K} (1\!-\!\kappa)^{K-k} C_{\thebetacounter}(\beta_{k}),
  \end{equation}
  for any positive integer $K$,
  where we denote $\kappa = \min\left\{\frac{\mu  R}{4 L (36 d+R)}, \frac{\sigma}{4} \right\}$ and
  \begin{align*}
    C_{\thebetacounter}(\beta_{k}) = \frac{\beta_k ^2 d R \left(d \left(32 \mu  \nu +36 L \sigma ^2+\mu  R \sigma ^2\right)+L R \sigma ^2\right)}{4 \mu  \sigma ^2 (36 d+R)^2}. 
  \end{align*}
  \end{theorem}

Theorem~\ref{tm:sc_main} establishes linear convergence of Algorithm~\ref{alg:zips} in the strongly convex case. However, the exact convergence rate still depends on $\sigma$ and $\nu$ which are further dependent on the specific distributions of $(\mathcal{R}_k,\mathcal{S}_k,\omega_k)$. For the special implementations mentioned in Section~\ref{section:algorithm}, we can establish more concrete convergence rates.

\begin{corollary}
  \label{cr:sc_main}
  Consider the implementations presented in Section~\ref{section:algorithm}. Suppose that the conditions of Theorem~\ref{tm:sc_main} hold, and that there is a constant $D>0$ such that $\|x^\ast\|_2\leq D$ and $L^{-1}\|\nabla f_i(x^\ast)\|_2\leq D$ for each $i$. Then, given any $\epsilon>0$, by letting $R\leq d$, $x_0=0$, $J_0=0$, and
  \[
  \alpha = \frac{R}{2 L (36 d+R)}, \quad
  \beta_{k} = \beta = \sqrt{\frac{R \mu^{2} \epsilon }{2 d^3 \mu  n^2 (36 L+\mu  n)+ 2d L^2 R}},\quad
  K=\Theta\!\left(\frac{d(n+L/\mu)}{R} \log\frac{n}{\epsilon}\right),
  \]
  we can achieve $\mathbb{E}[h(x_{K}) - h(x^\ast)] \leq \epsilon$. Consequently, the oracle complexity is bounded by
  \begin{equation*}
      O\!\left(d\!\left(n+\frac{L}{\mu}\right) \log\frac{n}{\epsilon}\right).
  \end{equation*}

\end{corollary}

As we can see, the oracle complexity $O\left(d(n+L/\mu) \log(n/\epsilon)\right)$ does not depend on the batch sizes $R$ . Thus in practice, we can use the purely incremental update scheme (Implementation~I in Section~\ref{subsection:impl}) with $R = 1$, for which one only needs to construct one $2$-point estimator per iteration. As a comparison, simply using the $2$-point zeroth-order gradient estimator without variance reduction can only achieve sublinear convergence. Moreover, our complexity bound is even better than the $O(dnL/\mu\cdot\log(1/\epsilon))$ complexity bound of zeroth-order full-batch algorithms when $n$ and the condition number $L/\mu$ are large.

Compared to the best known complexity bound $O\left((n+L/\mu) \log(1/\epsilon)\right)$ achieved by first-order variance reduction algorithms~\cite{defazio2014saga, johnson2013accelerating}, our proved complexity is only $O(d)$ times higher (neglecting terms logarithmic in $n$), bridging the gap between first-order and zeroth-order algorithms.

Among existing zeroth-order algorithms, ZPDVR \cite{di2024double} is the only method with guaranteed linear convergence on the composite finite-sum problem \eqref{eq:formulation} for strongly convex objective functions. Although their complexity bound $O\left(d(n+L/\mu) \log(1/\epsilon)\right)$ does not contain the $\log(n)$ term, their algorithm requires a $\Theta(n)$ batch sampling of the component functions from time to time. Moreover, their algorithm applies variance reduction twice---constructing one variance-reduced estimator on top of another---which makes their analysis considerably more involved. In contrast, our algorithm handles the two sources of variance in a single step, resulting in a simpler and more general method.

\subsection{Convex Case}
We then consider the case where each function $f_i$ is convex but $f$ is not assumed to be strongly convex. The following theorem summarizes the general convergence guarantees for Algorithm~\ref{alg:zips} in the convex case:
\begin{theorem}
  \label{tm:c_final}
\refstepcounter{betacounter}\label{counter:c_constant}
Suppose Assumptions~\ref{as:smooth} and \ref{as:algorithm_rand_var} hold, each $f_i$ is convex, and $x^\ast$ is a solution to~\eqref{eq:formulation}. Let the step size and the smoothing radii satisfy
\[
\alpha =\frac{R}{2L(40 d +R)}
\qquad\text{and}\qquad
\sum_{k=0}^{\infty} \beta_{k}<+\infty,
\]
with $\beta_k$ being non-increasing. Then we have
\refstepcounter{betacounter}\label{counter:c_final}
  \begin{equation*}
  \begin{aligned}
    \frac{1}{K} \sum_{k=1}^{K}\mathbb{E}[h(x_{k})] - h(x^{\ast})
\leq
  \frac{\sqrt{\mathrm{e}}(1\!+\!40d/R)L}{K}\!
    \left[2\lVert x_{0} \!-\! x^{\ast}\rVert_{2}^{2}
      + \frac{8dR}{L^2n\sigma(40d\!+\!R)^2}\lVert J_{0} \!-\! \nabla F(x^{\ast})\rVert_{F}^{2}
    + C_{\ref{counter:c_final}}\right]\!,
    \end{aligned}
  \end{equation*}
  for all positive integer $K$, where we denote $C_{\ref{counter:c_constant}}=\sum_{k=0}^{\infty} \beta_{k}$ and
  \begin{align*}
      C_{\ref{counter:c_final}} = 
    \frac{1}{2}
    \alpha L d C_2^2
    +
    \left(2+\frac{32\nu}{R\sigma^2}\right)\alpha^2 L^2 d^2\beta_0C_2.
  \end{align*}
\end{theorem}

The following corollary then provides oracle complexity bounds for the special implementations detailed in Section~\ref{section:algorithm}:

\begin{corollary}
  \label{cr:c_final}
  Consider the implementations presented in Section~\ref{section:algorithm}. Suppose that the conditions of Theorem~\ref{tm:c_final} hold, and that there is a constant $D>0$ such that $\|x^\ast\|_2\leq D$ and $L^{-1}\|\nabla f_i(x^\ast)\|_2\leq D$ for each $i$. 
  Then, given any $\epsilon>0$, by choosing $R\leq d$, $x_0=0$, $J_0=0$, and
  \[
  \alpha =\frac{R}{2L(40 d +R)},
  \quad
  \beta_k=\beta=O\!\left(\min\left\{
  \frac{\epsilon}{\sqrt{n}d},\frac{\sqrt{R\epsilon}}{nd^{3/2}}\right\}\right)\ \ \text{for }k\leq K,\quad
  K=\Theta\!\left(\frac{nd}{R\epsilon}\right),
  \]
  we can achieve achieve $\mathbb{E}[h(\bar{x}_{K}) - h(x^{\ast})] \leq \epsilon$ where $\bar{x}_K=\frac{1}{K}\sum_{k=1}^K x_k$. Consequently, the oracle complexity is bounded by
  \begin{equation*}
      O\left(\frac{nd}{\epsilon}\right).
  \end{equation*}
\end{corollary}

    We can see from Corollary~\ref{cr:c_final} that, in the convex case, the oracle complexity of Algorithm~\ref{alg:zips} can be independent of the batch size. Moreover, this complexity result, achieved by only two oracle queries per iteration, is the same as the complexity of the full-batch zeroth-order algorithm that requires $\Theta(nd)$ oracle queries per iteration. Although in the convex setting, variance reduction does not provide a better convergence rate compared to full-batch algorithms, our proposed algorithm has the clear advantage of taking $\Theta(1)$ instead of $\Theta(nd)$ oracle queries per iteration.

    We note that for the convex setting, the convergence results of both first-order and zeroth-order variance-reduced algorithms are scanty. Among the first-order ones, SAGA achieves an oracle complexity of $O(n/\epsilon)$, the same as first-order full-batch proximal gradient descent algorithm. The zeroth-order algorithm in~\cite{ji2019improved} achieves a complexity of $\displaystyle O\!\left(d \max\!\left\{n, \frac{1}{\epsilon}\right\} \ln\frac{1} {\epsilon}\right)$ but is designed for unconstrained problems. Moreover, the analysis in~\cite{ji2019improved} needs to explicitly impose boundedness of the iterates $x_k$ as an assumption in the unconstrained setting, which is nonstandard and seems to be a theoretical limitation.

\subsection{Non-convex Case}
When $f$ is non-convex, we introduce the standard gradient mapping~\cite{parikh2014proximal} $\mathfrak{g}_{\alpha}$ defined by
\begin{equation*}
  \mathfrak{g}_{\alpha}(x) = \frac{1}{\alpha}(x - \text{prox}_{\alpha \psi} (x - \alpha \nabla f(x)))
\end{equation*}
to measure the stationarity of a point. The quantity $\|\mathfrak{g}_{\alpha}(x)\|_2$ captures the size of a proximal-gradient step and therefore serves as a natural stationarity measure for composite optimization problems. Particularly, $x\mapsto \|\mathfrak{g}_{\alpha}(x)\|_2$ is continuous, and if $x$ is a local minimum point, then $\|\mathfrak{g}_{\alpha}(x)\|_2 = 0$. Hence we introduce the following convergence metric.
\begin{definition}
For each $\epsilon>0$, a point $x$ returned by \AlgAbbrv{} is called an $\epsilon$-accurate solution to the problem~\eqref{eq:formulation} if $\mathbb{E}\!\left[\lVert \mathfrak{g}_{\alpha}(x) \rVert_2^{2}\right] \leq \epsilon$ for some $\alpha > 0$.
\end{definition}
Now we are ready to present our main result.

\begin{theorem}
  \label{tm:nc_main}
Suppose Assumptions~\ref{as:smooth} and~\ref{as:algorithm_rand_var} hold, and $h^\ast\defeq \inf_{x\in\mathbb{R}^d} h(x)>-\infty$. Furthermore, suppose $R \leq d/\sigma^{2}$.
Then by setting the step size to be $\alpha = \frac{\sqrt{R} \sigma }{5 \sqrt{d} L}$,

we have
  \refstepcounter{betacounter}
  \label{counter:nc_final}
  \begin{equation}
    \label{eq:nc_main}
    \frac{1}{K}\sum_{k=0}^{K} \lVert \mathfrak{g}_{\alpha}(x_k) \rVert_{2}^{2} 
    \leq
    \frac{20 \sqrt{d} L}{K\sqrt{R} \sigma }
    \left(
    h(x_{0}) - h(x^{\ast}) + \frac{2\sqrt{d}}{5Ln\sqrt{R}}\|J_0-\nabla F(x_0)\|_F^2
    + \sum_{k=0}^{K} C_{\ref{counter:nc_final}}(\beta_{k})
    \right)
  \end{equation}
  for all positive integer $K$, where
  \begin{equation*}
    C_{\thebetacounter}(\beta_{k}) =
    \frac{\beta_k^2 d^{3/2} L \left(20 \nu +R \sigma ^2\right)}{20 \sqrt{R} \sigma }.
  \end{equation*}
\end{theorem}

Again we give the following corollary for the special implementations in Section~\ref{section:algorithm}.
\begin{corollary}
\label{corollary:nonconvex_complexity1}
  Consider the implementations presented in Section~\ref{section:algorithm}. Suppose that the conditions of Theorem~\ref{tm:nc_main} hold, and that there is a constant $V>0$ such that $h(0)-h(x^\ast)\leq V$ and $L^{-1}\|\nabla f_i(0)\|_2^2\leq V$ for each $i$. Then, given any $\epsilon>0$, by choosing $R\leq n^{2/3}d$, $x_0=0$, $J_0=0$, and
  \[
  \alpha = \frac{1}{5 nL}\!\left(\frac{R}{d}\right)^{\!3/2},
  \quad\beta_{k} = \beta = \frac{ \sqrt{\epsilon n }  R^{1/4}}{ 5\sqrt{d^{3/2} L \left( n^2+ R^2\right)}},
  \quad
  K = \Theta\!\left(\frac{n}{\epsilon}\!\left(\frac{d}{R}\right)^{\!3/2}\right),
  \]
  we can achieve $\mathbb{E}\!\left[\lVert\mathfrak{g}_{\alpha}(\widehat{x}_{K}) \rVert_{2}^{2}\right] \leq \epsilon$, where $\widehat{x}_K$ is uniformly chosen from $\{x_1,\ldots,x_K\}$ at random. Consequently, the oracle complexity to find an $\epsilon$-accurate solution is bounded by
  \begin{equation*}
    O\left(\frac{nd^{3/2}}{\epsilon \sqrt{R}}\right).
  \end{equation*}
Particularly, by letting $R=\lfloor n^{2/3}d\rfloor$, the corresponding oracle complexity is upper bounded by
\[
O\!\left(\frac{n^{2/3}d}{\epsilon}\right).
\]
  \end{corollary}

Corollary~\ref{corollary:nonconvex_complexity1} shows that, in the non-convex case, the overall oracle complexity bound is dependent on $R$, and is decreasing in $R$ until $R$ hits the upper bound $R\leq n^{2/3}d$ that results from the condition $R\leq d/\sigma^2$. Thus, if we set $R = 1$, the resulting complexity is still inferior to the $O(nd/\epsilon)$ oracle complexity of the full-batch zeroth-order algorithms. Fortunately, we can employ proper batch sampling $R=\lfloor n^{2/3}d\rfloor$ to get a better complexity bound.

In this case, the oracle complexity is better than the full-batch algorithm without variance reduction, while the number of oracle queries per iteration is $\Theta(n^{2/3} d)$, which is still less than the full-batch algorithm.
    
    The $O(n^{2/3}d/\epsilon)$ complexity is the same as that of the VR-SZD algorithm in~\cite{rando2025structured}, but their algorithm requires a maximum batch of size $\Theta(nd)$. The ZO-ProxSAGA algorithm~\cite{huang2019faster} uses a similar incremental update scheme as Implementation~I of our algorithm; however, their analysis assumes bounded gradient $\|\nabla f_i(x)\|_{2}^2 \leq \sigma^2$ for all $x$, and only convergence with limited precision $O(d\sigma^2)$ is established.
    
    Variance reduced algorithms that use biased estimators such as SPIDER~\cite{fang2018spider} and SARAH~\cite{nguyen2017sarah} are known to outperform those with unbiased estimators under non-convex settings. For zeroth-order problems, the SPIDER-SZO~\cite{fang2018spider} achieves a better oracle complexity of $O\!\left(d \min\left\{\frac{n^{1/2}}{\epsilon}, \frac{1}{\epsilon^{3/2}}\right\}\right)$; however, this result is restricted to unconstrained problems.
 \section{Convergence Analysis}
\label{section:proof}
In this section, we provide proof outlines for the theorems presented in Section~\ref{section:result}. We shall let $\mathcal{F}_k$ denote the $\sigma$-algebra generated by $(\omega_\tau,\mathcal{R}_\tau,\mathcal{S}_\tau)$ for all $\tau< k$. Note that $x_k$ and $J_k$ are $\mathcal{F}_{k}$-measurable.

We first present some auxiliary technical lemmas. The following lemma bounds the error of the zeroth-order estimate $G_{\beta_{k}}(x_k; \mathcal{S}_{k})$ with respect to $P_k (\nabla F(x_k) )$.
\begin{lemma}
  \label{lm:projected_bias}
For each $k$, we have
    \begin{equation}
    \label{eq:sc_projected_bias}
      \left\lVert G_{\beta_{k}}(x_k; \mathcal{S}_{k}) - P_k(\nabla F(x_k)) \right\rVert^{2}_{F}
      \leq \frac{ \lvert \mathcal{S}_{k}\rvert L^{2} \beta_{k}^{2}}{4}.
    \end{equation}
\end{lemma}

Then, the next lemma shows that our gradient estimator $g_k$ is approximately unbiased.

\begin{lemma}
  \label{lm:estimator_bias}
  For each $k$, we have
    \begin{equation}
      \label{eq:sc_estimator_bias}
     \left\lVert \mathbb{E}\Econd{g_{k}}{\mathcal{F}_{k}} - \nabla f(x_{k})\right\rVert_{2}^{2}
      \leq \frac{d L^{2} \beta_k^{2}}{4}.
    \end{equation}
\end{lemma}

The proofs of Lemmas~\ref{lm:projected_bias} and~\ref{lm:estimator_bias} are postponed to Appendix~\ref{appendix:auxiliary}.

We also need the following lemma that summarizes some well-known properties for smooth and convex functions.

\begin{lemma}
  \label{lm:smoothness}
Suppose Assumption~\ref{as:smooth} holds and each $f_i$ is convex. Then for all $x,y\in\mathbb{R}^d$, we have
\begin{equation}
\label{eq:sc_smoothness}
      \langle \nabla f(x) - \nabla f(y), x-y\rangle
      \geq \frac{1}{nL} \lVert \nabla F(x) - \nabla F(y)\rVert^{2}_{2},
    \end{equation}
  \begin{equation}
  \label{eq:sc_bregman_bound}
    f(x) - f(y) - \langle \nabla f(y), x-y\rangle  
    \geq \frac{1}{2nL}\|\nabla F(x) - \nabla F(y)\|_F^2.
  \end{equation}
Furthermore, if $f$ is $\mu$-strongly convex, then for all $x,y\in\mathbb{R}^d$,
\begin{equation}
\label{eq:sc_inner_prod}
\langle\nabla f(x)-\nabla f(y),x-y\rangle
\geq\frac{\mu}{2}\|x-y\|_2^2+\frac{1}{2nL}\|\nabla F(x)-\nabla F(y)\|_F^2.
\end{equation}
\end{lemma}
\begin{proof}
The inequalities ~\eqref{eq:sc_smoothness} and~\eqref{eq:sc_bregman_bound} directly follow by applying {\cite[Eq. (2.1.10) and (2.1.11)]{nesterov2018lectures}} to each $f_i$, taking the sum and noting that $\sum_{i=1}^n\lVert\nabla f_i(x)-\nabla f_i(y)\rVert^2=\lVert\nabla F(x)-\nabla F(y)\rVert_F^2$. The last inequality follows by adding~\eqref{eq:sc_bregman_bound} with
\[
f(y)\geq f(x)+\langle \nabla f(x),y-x\rangle+\frac{\mu}{2}\lVert y-x\rVert_2^2,
\]
which follows from the $\mu$-strong convexity of $f$.
\end{proof}

 \subsection{Analysis for the Strongly Convex Case}

In this subsection, we shall always assume that Assumptions~\ref{as:smooth} and~\ref{as:algorithm_rand_var} hold, each $f_i$ is convex, and $f$ is $\mu$-strongly convex.

We start by giving a bound on the expected deviation of the gradient estimator $g_{k}$ from the gradient at the optimal point $x^{\ast}$. \begin{lemma}
  \label{lm:sc_gradient_bound}
For each $k=0,1,2,\ldots$, we have
  \begin{equation}
    \mathbb{E}\Econd*{\|g_k - \nabla f(x^{\ast})\|_2^2}{\mathcal{F}_{k}}
    \leq
    \left(\frac{2}{n}\!+\! \frac{8d}{n R}\right)\left\| \nabla F(x_{k}) - \nabla F(x^{\ast})\right\|_{F}^2
    + \frac{8d}{nR} \| J_k -\nabla F(x^{\ast}) \|_F^2
    + d^2 L^2 \beta_k^2.
   \label{eq:sc_gradient_bound}
  \end{equation}
    
\end{lemma}

We will give a recursive bound on $\mathbb{E}\Econd*{\|J_{k+1} - \nabla F(x^{\ast})\|_F^2}{\mathcal{F}_{k}}$ as well:

\begin{lemma}
For each $k=0,1,2,\ldots$, we have
  \label{lm:sc_jacobian_recur}
  \begin{equation}
    \mathbb{E}\Econd*{\|J_{k+1} - \nabla F(x^{\ast})\|_F^2}{\mathcal{F}_{k}}
    \leq \left(1\!-\! \frac{\sigma}{2}\right)\|J_k - \nabla F(x^{\ast})\|_{F}^{2}
    +2 \sigma\| \nabla F(x_k)  -\nabla F(x^{\ast})\|_{F}^{2} 
    +  \frac{ \nu nd L^{2} \beta_{k}^{2}}{\sigma}.
    \label{eq:sc_jacobian_recur}
  \end{equation}
\end{lemma}

The proofs of the two lemmas are postponed to Appendix~\ref{appendix:proof_strongly_convex}.

Now we are ready to prove our main result.

\begin{proof}[Proof of Theorem~\ref{tm:sc_main}]
  We start by the non-expansiveness of the proximal operator
    \begin{align*}
      \mathbb{E}\Econd*{\|x_{k+1}-x^*\|_{2}^{2}}{\mathcal{F}_{k}} 
      &= \mathbb{E}\Econd*{\|\operatorname{prox}_{\alpha R}(x_{k}-\alpha g_{k})-\operatorname{prox}_{\alpha R}(x^*-\alpha\nabla f(x^*))\|_{2}^{2}}{\mathcal{F}_k} \\
      &\leq \mathbb{E}\Econd*{\|x_{k}-\alpha g_{k}-(x^*-\alpha\nabla f(x^*))\|_2^{2}}{\mathcal{F}_k} \\
      &= \|x_{k}-x^*\|_{2}^{2}
      -2\alpha\mathbb{E}\Econd*{\langle g_{k}-\nabla f(x^\ast), x_{k}-x^\ast\rangle }{\mathcal{F}_k} \\
      &\quad +\alpha^{2}\mathbb{E}\Econd*{\|g_{k}-\nabla f(x^*)\|_{2}^{2}}{\mathcal{F}_k}. \numberthis \label{eq:sc_nonexpan}
\end{align*}
For the second term, we give the following bound
\begin{align*}
      & -2\alpha\mathbb{E}\Econd{\langle g_{k}-\nabla f(x^\ast), x_{k}-x^\ast\rangle }{\mathcal{F}_k} \\
      ={} & -2\alpha\langle \nabla f(x_{k})-\nabla f(x^\ast), x_{k}-x^\ast\rangle
      + 2\alpha \langle \nabla f(x_{k})-\mathbb{E}\Econd{g_k}{\mathcal{F}_k}, x_{k}-x^\ast\rangle\\
      \stackrel{\eqref{eq:sc_inner_prod}}{\leq}{} &-\alpha \mu \lVert x_{k} - x^{\ast}\rVert^{2}_{2}
      - \frac{\alpha}{nL} \lVert \nabla F(x_k) - \nabla F(x^{\ast})\rVert_{F}^{2}
      + 2\alpha \langle \nabla f(x_{k})-\mathbb{E}\Econd{g_k}{\mathcal{F}_k}, x_{k}-x^\ast\rangle
      \numberthis \label{eq:sc_4convex}
      \\ 
      \leq{} &
      -\frac{\alpha \mu }{2} \lVert x_{k} - x^{\ast}\rVert^{2}_{2}
      - \frac{\alpha}{nL} \lVert \nabla F(x_k) - \nabla F(x^{\ast})\rVert_{F}^{2} 
      + \frac{2 \alpha}{\mu} \lVert \nabla f(x_{k})-\mathbb{E}\Econd{g_k}{\mathcal{F}_k}\rVert_{2}^{2}
      \\ 
      \stackrel{\eqref{eq:sc_estimator_bias}}{\leq}{} &
       -\frac{\alpha \mu }{2} \lVert x_{k} - x^{\ast}\rVert^{2}_{2}
      - \frac{\alpha}{nL} \lVert \nabla F(x_k) - \nabla F(x^{\ast})\rVert_{F}^{2} 
      + \frac{ d L^{2} \beta_{k}^{2}  \alpha}{2 \mu},
      \numberthis \label{eq:sc_cross_term}
\end{align*}
where the third step follows from Young's inequality.
For the last term of~\eqref{eq:sc_nonexpan}, we can apply Lemma~\ref{lm:sc_gradient_bound}, which leads to
\begin{align*}
      &\mathbb{E}\Econd*{\|x_{k+1}-x^*\|_{2}^{2}}{\mathcal{F}_{k}} \\
      \leq{} & \left(1- \frac{\alpha \mu}{2}\right) \|x_{k}-x^*\|_{2}^{2} 
      + \frac{\alpha}{n}\left( 2 \alpha + \frac{8d \alpha}{R}- \frac{1}{L}\right)
      \lVert \nabla F(x_k) - \nabla F(x^{\ast})\rVert_{F}^{2}\\
      & + \frac{8 d \alpha^{2}}{n R}
      \| J_{k} - \nabla F(x^{\ast}) \|_F^2 
      + L^{2} \beta_{k}^{2} d \alpha \left( \frac{1}{2 \mu} + \alpha d\right)
      \numberthis \label{eq:sc_distance_contra}.
\end{align*}
Now we define the Lyapunov function
\begin{equation*}
  \Phi_{1}^{k} = \lVert x_k - x^{\ast}\rVert^{2}_{2} + c_{1} \lVert J_k - \nabla F(x^{\ast})\rVert^{2}_{F},
\end{equation*}
where
$
c_{1} = \frac{8 d R}{L^2 n \sigma  (36 d+R)^2}$.
Using~\eqref{eq:sc_distance_contra} together with Lemma~\ref{lm:sc_jacobian_recur}, we can give the following bound for $\Phi_{1}^{k+1}$:
\stepcounter{betacounter}
\begin{align*}
  \mathbb{E}\Econd*{\Phi_{1}^{k+1}}{\mathcal{F}_k}
  &=
  \mathbb{E}\Econd*{\lVert x_{k+1} - x^{\ast}\rVert^{2}_{2}}{\mathcal{F}_k}
  +c_{1}\mathbb{E}\Econd*{\lVert J_{k+1}- \nabla F(x^{\ast})\rVert^{2}_{F}}{\mathcal{F}_{k}}
  \\ &\leq 
  \left(1 - \frac{\alpha \mu}{2}\right) \|x_{k}-x^*\|_{2}^{2} 
  + \left( \frac{ 2 \alpha^{2}}{n} + \frac{8 d \alpha^{2}}{nR} - \frac{\alpha}{nL} + 2 c_{1} \sigma\right)
  \lVert \nabla F(x_k) - \nabla F(x^{\ast})\rVert_{F}^{2}\\
  &\quad + \left(\frac{8 d\alpha^{2}}{nR} + c_{1}\left(1-\frac{\sigma}{2}\right)\right)
  \| J_{k} - \nabla F(x^{\ast}) \|_F^2 
  + C_{\ref{counter:sc_final}}(\beta_{k}).
\end{align*}
Now with our parameter choice of $c_{1}$ and $\alpha$, we know that
\[
\frac{ 2 \alpha^{2}}{n} + \frac{8 d \alpha^{2}}{nR} - \frac{\alpha}{nL} + 2 c_{1} \sigma= 0,
\quad \text{and}\quad
\frac{8 d\alpha^{2}}{nR} + c_{1}\left(1-\frac{\sigma}{2}\right) = \left(1- \frac{\sigma}{4}\right) c_{1}.
\]
Thus
\begin{align*}
  \mathbb{E}\Econd*{\Phi_{1}^{k+1}}{\mathcal{F}_k}
  &\leq
  \left(1- \frac{\alpha \mu}{2}\right) \|x_{k}-x^*\|_{2}^{2} 
  + \left(1- \frac{\sigma}{4}\right)c_{1} 
  \| J_{k} - \nabla F(x^{\ast}) \|_F^2 + C_{\ref{counter:sc_final}}(\beta_{k})
  \\&\leq
   (1- \kappa) \Phi_{1}^{k} + C_{\ref{counter:sc_final}}(\beta_{k}).
\end{align*}
By taking full expectation and telescoping, we obtain
\begin{align*}
  \mathbb{E}\!\left[\Phi_{1}^{K}\right]
  &\leq 
  (1-\kappa)^{K} \Phi_{1}^{0}
  + \sum_{k=0}^{K-1} (1-\kappa_{1})^{k} C_{\ref{counter:sc_final}}(\beta_{k})
\\&=
  (1-\kappa)^{K} 
  \left( \lVert x_{0} - x^{\ast} \rVert_2^2 
    +  c_{1} \lVert J_0 - \nabla F(x^{\ast})\rVert_{F}^{2}\right)
  + \sum_{k=1}^{K} (1-\kappa)^{K-k} C_{\ref{counter:sc_final}}(\beta_{k-1})
  \numberthis \label{eq:sc_lyapunov_recur}.
\end{align*}

Using $\mathbb{E}[\lVert x_k - x^{\ast}\rVert_{2}^{2}] 
  \leq
  \mathbb{E}[\Phi_{1}^{K}]$, we get the desired bound~\eqref{counter:sc_final}.
\end{proof}

 \subsection{Analysis for the Convex Case}
Throughout this subsection, we assume that Assumption~\ref{as:smooth} holds, and each $f_i$ is convex. To start our analysis for the convex case, 
we shall establish a recursive bound on the distance from the iterate $x_k$ to the optimal point; its proof can be found in Appendix~\ref{sec:app_proof_lm_c_alt_distance}.
\begin{lemma}
\label{lm:c_alt_distance}
For each iteration $k$, we have
    \begin{align*}
  \mathbb{E}\Econd*{\lVert 
    x_{k+1} \!-\! x^{\ast}
  \rVert_{2}^{2}}{\mathcal{F}_k}
  \leq{} &
  \!\left(1\!+\! \frac{\alpha L \beta_k}{C_2}\right)\! \lVert 
    x_{k} \!-\! x^{\ast}
  \rVert_{2}^{2} 
  - 2 \alpha\,\mathbb{E}\Econd*{
    h(x_{k+1}) \!-\! h(x^{\ast})
  }{\mathcal{F}_{k}}
  + \frac{8 \alpha^2 d}{nR } \| J_{k} \!-\! \nabla F(x^{\ast})\|_F^2 
       \\ & +\frac{8 \alpha^2 d}{n R} \| \nabla F(x_k) \!-\! \nabla F(x^{\ast})\|_F^2  + \alpha L\beta_kd \left(\frac{C_2}{4}+ \alpha L \beta_kd\right) . 
         \numberthis \label{eq:c_alt_distance}
  \end{align*}
\end{lemma}

Now we are ready to prove Theorem~\ref{tm:c_final}.

\begin{proof}[Proof of Theorem~\ref{tm:c_final}]
We begin from~\eqref{eq:sc_4convex} with $\mu = 0$:
  \begin{align*}
      & -2\alpha\mathbb{E}\Econd*{\langle g_{k}-\nabla f(x^\ast), x_{k}-x^\ast\rangle }{\mathcal{F}_k}
      \\ \leq{} & 
      - \frac{\alpha}{nL} \lVert \nabla F(x_{k}) - \nabla F(x^{\ast})\rVert_{F}^{2}
      + 2\alpha \langle \nabla f(x_{k})-\mathbb{E}\Econd*{g_{k}}{\mathcal{F}_k}, x_{k}-x^\ast\rangle
      \\ \leq{} & 
      - \frac{\alpha}{nL} \lVert \nabla F(x_{k}) - \nabla F(x^{\ast})\rVert_{F}^{2}
      + \frac{\alpha L \beta_{k} C_2 d }{4}
      + \frac{\alpha L \beta_{k}}{C_2}
      \lVert x_{k} - x^{\ast}\rVert_{2}^{2},
  \end{align*}
  where the last inequality comes from Young's inequality and~\eqref{eq:sc_estimator_bias}. Now we can plug the above inequality and~\eqref{eq:sc_gradient_bound} into~\eqref{eq:sc_nonexpan} to get
  \begin{align*}
      & \mathbb{E}\Econd*{\|x_{k+1}-x^*\|_{2}^{2}}{\mathcal{F}_{k}} \\
      \leq{} & \left(1+  \frac{\alpha L \beta_{k}}{C_2}\right) \|x_{k}-x^*\|_{2}^{2} 
      + \frac{\alpha}{n}\left( 2\alpha + \frac{8\alpha d}{R}- \frac{1}{L}\right)
      \lVert \nabla F(x_{k}) - \nabla F(x^{\ast})\rVert_{F}^{2}\\
      & + \frac{8 d\alpha^{2}}{n R}
      \| J_k - \nabla F(x^{\ast}) \|_F^2 
      + \alpha L \beta_{k} d \left( \frac{C_2}{4 } + \alpha L\beta_k d \right).
      \numberthis \label{eq:c_distance_contra}
  \end{align*}
  
  Now we define a new Lyapunov function:
  \begin{equation*}
    \Phi_{2}^{k} = 2\lVert x_{k} - x^{\ast}\rVert^{2}_{2} + c_{2} \lVert J_{k} - \nabla F(x^{\ast})\rVert^{2}_{F},
  \end{equation*}
  with $c_{2} = \frac{32\alpha^2d}{Rn\sigma}=\frac{8 d R}{L^2 n \sigma  (40 d+R)^2}$. We can now add~\eqref{eq:c_alt_distance}, \eqref{eq:c_distance_contra} and~\eqref{eq:sc_jacobian_recur} multiplied by $c_{2}$ to get
  \begin{align*}
    \mathbb{E}\Econd*{\Phi_{2}^{k+1}}{\mathcal{F}_{k}}
  \leq{} & 2\!\left(1+\frac{\alpha L \beta_{k}}{C_2}\right) \|x_{k}-x^*\|_{2}^{2} 
  + \left(\frac{2 \alpha^{2}}{n} + \frac{16 d \alpha^{2}}{n R} - \frac{\alpha}{nL} + 2 c_{2} \sigma\right)
  \lVert \nabla F(x_{k}) - \nabla F(x^{\ast})\rVert_{F}^{2}
  \\ 
  &
  + \left(\frac{16 d \alpha^{2}}{nR} + c_{2} \left(1 - \frac{\sigma}{2}\right)\right)
  \| J_{k} - \nabla F(x^{\ast}) \|_F^2 
  - 2 \alpha \mathbb{E}\Econd*{
    h(x^{k+1}) - h(x^{\ast})
  }{\mathcal{F}_{k}}
+ C_{\thebetacounter}(\beta_{k}),
  \end{align*}
  where
  \[
  C_{\thebetacounter}(\beta_{k}) 
  \defeq \frac{\alpha LdC_2}{2}\beta_k
  +\left(2+\frac{32\nu}{R\sigma^2}\right)\alpha^2 L^2 d^2 \beta_k^2.
  \]
  With our choices of the algorithm parameters, we have
  \begin{gather*}
  \frac{2 \alpha^{2}}{n} + \frac{16 d \alpha^{2}}{n R} - \frac{\alpha}{nL} + 2 c_{2} \sigma = 0,\qquad
      \frac{16 d \alpha^{2}}{nR} + c_{2} \left(1 - \frac{\sigma}{2}\right) = c_{2},
  \end{gather*}
  and therefore
  \begin{equation*}
    \mathbb{E}\Econd*{\Phi_{2}^{k+1}}{\mathcal{F}_{k}} \leq \left(1+ \frac{\alpha L \beta_{k}}{C_2}\right) \Phi_{2}^{k}
  - 2 \alpha \mathbb{E}\Econd*{
    h(x^{k+1}) - h(x^{\ast})
  }{\mathcal{F}_{k}}
  + C_{\thebetacounter}(\beta_{k}).
  \end{equation*}
  Taking total expectations on both sides and denoting $\gamma_{k} \defeq \Pi_{l=k}^{K-1} (1+ \alpha L \beta_{l}/C_2)$, we get
  \begin{equation*}
    \gamma_{k+1} \mathbb{E}[\Phi_{2}^{k+1}] 
    \leq \gamma_{k} \mathbb{E}[\Phi_{2}^{k}]
    - 2 \alpha \gamma_{k+1} \mathbb{E}[
    h(x^{k+1}) - h(x^{\ast})]
  + \gamma_{k+1} C_{\thebetacounter}(\beta_{k}).
  \end{equation*}
  Telescoping sum gives
   \begin{align*}
    \mathbb{E}[\Phi_{2}^{K}] 
    &\leq 
    \gamma_{0} \mathbb{E}[\Phi_{2}^{0}]
    +
    \sum_{k=0}^{K-1} \gamma_{k+1} \left(
    - 2 \alpha \mathbb{E}[
    h(x^{k+1}) - h(x^{\ast})]
  + C_{\thebetacounter}(\beta_{k})\right)
    \\&\leq 
    \gamma_{0} \mathbb{E}[\Phi_{2}^{0}]
    +
    \sum_{k=0}^{K-1} \left(
    - 2 \alpha \mathbb{E}[
    h(x^{k+1}) - h(x^{\ast})]
  + \gamma_{k+1} C_{\thebetacounter}(\beta_{k})\right).
  \end{align*}   
  We also know that for all $k, K \in \mathbb{N}$ such that $k < K$, we have
  \begin{equation*}
    \ln\gamma_{k} = \sum_{l=k}^{K-1} \ln\!\left(1+ \frac{\alpha L \beta_{l}}{C_2}\right) < \sum_{l=k}^{K-1} \frac{\alpha L \beta_{l}}{C_2} < \sum_{l=0}^{\infty} \frac{\alpha L\beta_{l}}{C_2} \leq \alpha L
    \leq\frac{1}{2}.
  \end{equation*}
  Therefore
  \begin{align*}
    \sum_{k=0}^{K-1} 
    \gamma_{k+1} C_{\thebetacounter}(\beta_{k})
    &\leq
    \sqrt{\mathrm{e}} 
    \sum_{k=0}^{\infty} C_{\thebetacounter}(\beta_{k})
    = \sqrt{\mathrm{e}} 
    \sum_{k=0}^{\infty} 
    \left[
  \frac{\alpha LdC_2}{2}\beta_k
  +\left(2+\frac{32\nu}{R\sigma^2}\right)\alpha^2 L^2 d^2 \beta_k^2
    \right]
    \\&\leq
    \frac{\sqrt{\mathrm{e}}}{2}
    \alpha L d C_2^2
    +
    \sqrt{\mathrm{e}}\!
    \left(2+\frac{32\nu}{R\sigma^2}\right)\alpha^2 L^2 d^2\beta_0C_2
  =\sqrt{\mathrm{e}}C_3.
  \end{align*}
  
  Finally, rearranging terms and plugging in our choice of $\alpha$ and $c_{2}$, we have
  \begin{align*}
    &\quad \frac{1}{K} \sum_{k=1}^{K}\mathbb{E}[h(x_{k})] - h(x^{\ast}) 
  \leq
    \frac{  \sqrt{\mathrm{e}}\Phi_{2}^{0}- \mathbb{E}[\Phi_{2}^{K}] }{2\alpha K} + \frac{\sqrt{\mathrm{e}} C_3}{2 \alpha K}
  \leq
    \frac{\sqrt{\mathrm{e}}L(1+40d/R)}{K}(\Phi_{2}^{0} + C_3)
  \\&\leq
  \frac{\sqrt{\mathrm{e}}L(1+40d/R)}{K}
    \left(2  \lVert x_{0} - x^{\ast}\rVert_{2}^{2}
      + \frac{8 d R}{L^2 n \sigma  (40 d+R)^2} \lVert J_0 - \nabla F(x^{\ast})\rVert_{F}^{2}
    + C_3\right),
    \end{align*}
which completes the proof.
\end{proof}

 \subsection{Analysis for the Non-convex Case}
\stepcounter{betacounter}
For the non-convex case, we need to give a different bound on the variance of $g_{k}$:
\begin{lemma}
\label{lm:nc_grad_var_bound}
For each $k$, we have
    \begin{equation}
      \label{eq:nc_grad_var_bound}
      \mathbb{E}\Econd*{\lVert g_{k} - \nabla f(x_{k})\rVert_{2}^{2}}{\mathcal{F}_{k}} 
      \leq \frac{2d}{n R} \| \nabla F(x_k) - J_k\|_F^2
   + \frac{d^2 L^2 \beta_k^2}{2}.
    \end{equation}
\end{lemma}

In addition, we also need a bound on the mean square error of the Jacobian estimator:
\begin{lemma}
\label{lm:nc_jacobian_error}
For each $k$, we have
\begin{align}\label{eq:nc_jacobian_error}
    &\quad \mathbb{E}\Econd*{\lVert \nabla F(x_{k+1}) - J_{k+1}\rVert_{F}^{2}}{\mathcal{F}_{k}}
    \nonumber \\ &\leq
    \left(1- \frac{\sigma}{2}\right) \lVert \nabla F(x_{k}) - J_{k})\rVert_{F}^{2}
      + \frac{5n L^{2} }{\sigma} \mathbb{E}\Econd*{\lVert x_{k+1} - x_{k}\rVert_{2}^{2}}{\mathcal{F}_k} 
      +  \frac{5 \nu n d L^{2} \beta_{k}^{2} }{2 \sigma}.
  \end{align}
\end{lemma}

The following lemma characterizes the evolution of the expected cost function value:
\begin{lemma}
  \label{lm:nc_frac_g}
  With our definition of $\mathfrak{g}_{\alpha}(x)$, at the $k$-th iteration, we have
  \begin{align}\label{eq:nc_frac_g}
    \mathbb{E}\Econd*{h(x_{k+1})}{\mathcal{F}_{k}}
    \leq{} & h(x_{k}) 
    + \left(\frac{L}{2} - \frac{1}{2 \alpha}\right)
    \mathbb{E}\Econd*{\lVert x_{k+1} - x_{k}\rVert_{2}^{2}}{\mathcal{F}_{k}}
    + \left(\alpha^{2}L - \frac{\alpha}{2}\right) \lVert \mathfrak{g}_{\alpha}(x_{k})\rVert_{2}^{2}.
  \nonumber \\&
  + \frac{\alpha}{2} \mathbb{E}\Econd*{\lVert g_{k} - \nabla f(x_{k})\rVert^{2}_{2}}{\mathcal{F}_{k}}.
  \end{align}
\end{lemma}

The proofs of these lemmas are postponed to Appendix~\ref{app:proof_nonconvex}
Now we are ready to prove Theorem~\ref{tm:nc_main}.
\begin{proof}[Proof of Theorem~\ref{tm:nc_main}]
  We define the Lyanpunov function $\Phi_{3}^{k}$ as
  \begin{equation*}
      \Phi_{3}^{k} := 
    h(x_{k}) + c_{3} \lVert \nabla F(x_{k}) - J_{k}\rVert_{F}^{2},
  \end{equation*}
  where $c_{3} = \frac{2 \alpha  d}{n R \sigma }$.
  Combining \eqref{eq:nc_jacobian_error} and \eqref{eq:nc_frac_g} gives
  \begin{align*}
  \mathbb{E}\Econd*{
    \Phi_{3}^{k+1}
  }{\mathcal{F}_{k}}
={} &
  \mathbb{E}\Econd*{
    h(x_{k+1}) + c_{3} \lVert \nabla F(x_{k+1} - J_{k+1})\rVert_{F}^{2}
  }{\mathcal{F}_{k}}\\
  \leq{} & h(x_{k})
  + \frac{\alpha}{2} \mathbb{E}\Econd*{\lVert
    g_{k} - \nabla f(x_{k})
  \rVert^{2}_{2}}{\mathcal{F}_{k}}
  + \left(\frac{L}{2} - \frac{1}{2 \alpha} + \frac{5c_{3}nL^{2}}{\sigma}\right)
  \mathbb{E}\Econd*{\lVert x_{k+1} - x_{k}\rVert_{2}^{2}}{\mathcal{F}_{k}}\\
  & + c_{3}\left(1- \frac{ \sigma}{2}\right) \lVert 
  \nabla F(x_{k}) - J_{k}
  \rVert_{F}^{2}
  + \left(\alpha^{2}L - \frac{\alpha}{2}\right) \lVert \mathfrak{g}_{\alpha}(x_{k}) \rVert_{2}^{2}+ \frac{5 c_{3} \nu n d L^{2} \beta_{k}^{2} }{2 \sigma}\\
  \leq{} & h(x_{k})
  + \left(\frac{L}{2} - \frac{1}{2 \alpha} + \frac{5c_{3}nL^{2}}{\sigma}\right)
  \mathbb{E}\Econd*{\lVert x_{k+1} - x_{k}\rVert_{2}^{2}}{\mathcal{F}_{k}} \\
  & + \left(c_{3}\left(1- \frac{ \sigma}{2}\right) + \frac{d \alpha}{n R}\right) \lVert 
  \nabla F(x_{k}) - J_{k}
  \rVert_{F}^{2}
  + \left(\alpha^{2}L - \frac{\alpha}{2}\right) \lVert \mathfrak{g}_{\alpha}(x_{k}) \rVert_{2}^{2}
  + C_{\ref{counter:nc_final}}(\beta_{k}),
  \end{align*}
  where in the last inequality we used \eqref{eq:nc_grad_var_bound}.
  With our choice of $c_{3}$ and $\alpha$ and the condition $R\leq d/\sigma^2$, we have
  \begin{equation*}
  c_{3}\left(1- \frac{ \sigma}{2}\right) + \frac{d \alpha}{n R} = c_{3},\quad \frac{L}{2} - \frac{1}{2 \alpha} + \frac{5c_{3}nL^{2}}{\sigma} \leq 0.
  \end{equation*}
  So we can drop the non-positive term and get
  \begin{equation*}
    \mathbb{E} \Econd*{\Phi_{3}^{k+1}}{\mathcal{F}_{k}}
    \leq
    \Phi_{3}^{k} - \left(\frac{\alpha}{2 } - \alpha^{2}L\right)\lVert \mathfrak{g}_{\alpha}(x_{k}) \rVert_{2}^{2} 
    + C_{\ref{counter:nc_final}}(\beta_{k}).
  \end{equation*}
  Taking full expectation and telescoping sum yields
  \begin{align*}
    \sum_{k=0}^{K} 
    \mathbb{E}\!\left[\lVert \mathfrak{g}_{\alpha}(x_{k}) \rVert_{2}^{2}\right]
    &\leq
    \frac{2}{\alpha - 2 \alpha^{2} L}
    \left(h(x_{0}) - h(x_{k}) + c_{3} \lVert \nabla F(x_{0}) - J_0\rVert_{2}^{2} 
    + \sum_{k=0}^{K} C_{\ref{counter:nc_final}}(\beta_{k})\right) \\
    &=
    \frac{2}{\alpha - 2 \alpha^{2} L}
    \left(h(x_{0}) - h(x^\ast) + 
    \frac{2\sqrt{d}}{5Ln\sqrt{R}}\lVert \nabla F(x_{0}) - J_0\rVert_{2}^{2} 
    + \sum_{k=0}^{K} C_{\ref{counter:nc_final}}(\beta_{k})\right).
  \end{align*}
  We also notice that by our choice of $\alpha$ and requirement on $R$, we have $\frac{2}{\alpha - 2 \alpha^{2} L} \leq \frac{4}{\alpha} = \frac{20 \sqrt{d} L}{\sqrt{R} \sigma }$. We now obtain the desired bound.
\end{proof}
 \section{Numerical Experiments}
  In this section, we test our 2-point \AlgAbbrv{} algorithm against other zeroth-order algorithms, including ZPDVR~\cite{di2024double}, stochastic proximal zeroth-order gradient descent algorithm without variance reduction (here referred to as Vanilla ZO)~\cite{duchi2015optimal}, a full-batch zeroth-order gradient descent algorithm and ZPSVRG, which is the Prox-SVRG method~\cite{xiao2014proximal} adapted to zeroth-order settings. The source code for these experiments is available at https://github.com/zhangsilan/ZIVR.
  \subsection{Logistic Regression}\label{section:exp_logistic}
  We first consider the following logistic regression problem with elastic net regulariztion~\cite{zou2005elastic}:
  \begin{equation}\label{eq:exp_logistic}
    \min_{x \in \mathbb{R}^{d}}\frac{1}{n}\sum_{i=1}^{n}\ln(1+\exp(-b_i a_i ^T x)) + \frac{\mu}{2}\|x\|_2^2 + \lambda \|x\|_1,
  \end{equation}
  whith $a_{i}$ being individual feature vectors $b_{i}$ the corresponding label. This problem can be seen as a special example of~\eqref{eq:intro_formulation} with $f_{i}(x) = \ln(1+\exp(-b_i a_i ^T x)) + \frac{\mu}{2}\|x\|_2^2 $, and $\psi(x) = \lambda \lVert x\rVert_{1}$. We conduct our experiment on two data sets from the LIBSVM project~\cite{Chang2011libsvm}, the parameters of these two problems are listed in Table~\ref{tb:dataset}.
  \begin{table}
    \centering
  \caption{Data sets and regularization coefficients used in our experiments.}    \label{tb:dataset}
  \begin{tabular}{c c c c c}
    \toprule
    Data Set & $n$ &  $d$ & $\mu$ & $\lambda$\\
    \midrule
    a9a & $32561$ & $123$ & $10^{-4}$ & $10^{-4}$\\
    w8a & $49749$ & $300$ & $10^{-4}$ & $10^{-4}$\\
    \bottomrule
  \end{tabular}    
  \end{table}
\begin{figure}
  \centering
  \subfloat[\centering a9a]{\includegraphics[width=0.49\linewidth]{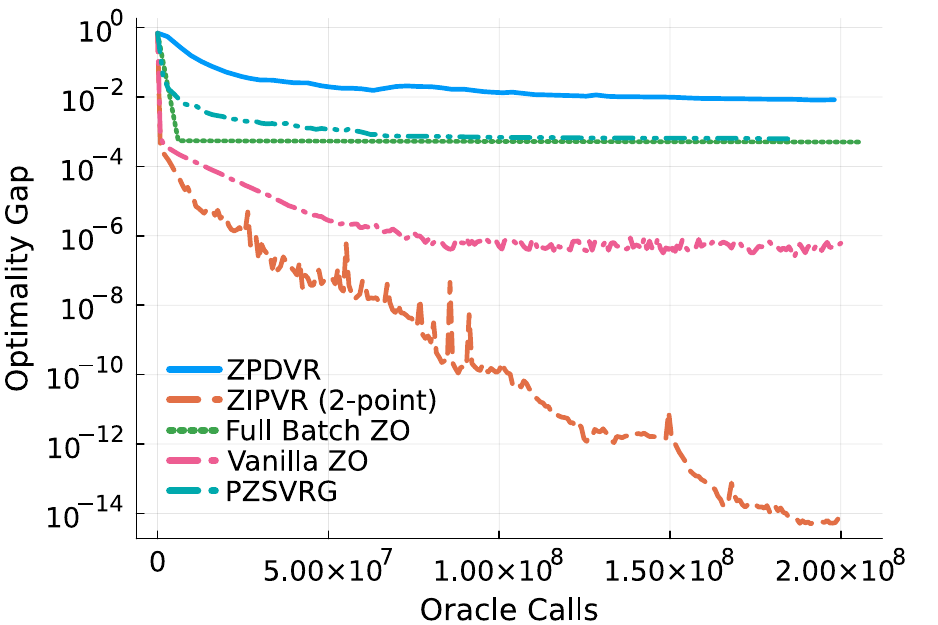}}
  \hfill
  \subfloat[\centering w8a]{\includegraphics[width=0.49\linewidth]{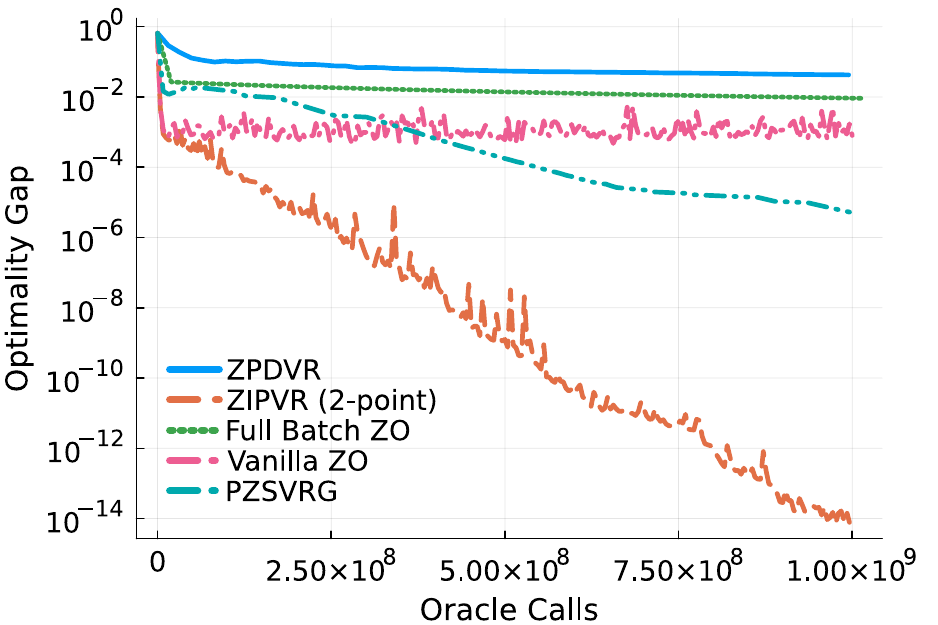}}
  \caption{Comparison of different zeroth-order methods on logistic regression problems, the $x$-axis is the number of zeroth-order oracle calls and the $y$-axis is the optimality gap $h(x^{k})- h(x^{\ast})$.}
    \label{fig:exp1}
\end{figure}

The results of our experiments are shown in Figure~\ref{fig:exp1}, we can see that our 2-point \AlgAbbrv{} algorithm outperforms other algorithms by a large margin.  The vanilla zeroth-order algorithm without variance reduction performs well at first, its curve soon flattens, this further suggests that non-vanishing variance may hinder the algorithm from converging to the optimal point. The ZPDVR algorithm, though also designed to deal with the two-fold variance in this problem and exhibits a linear convergence pattern, also converges significantly slower than \AlgAbbrv{}. As for the full-batch zeroth-order algorithm, the curve is generally very flat except for the first iteration, since the $x$-axis signifies the oracle calls instead of iteration number, and for the full-batch algorithm these seemingly large budgets of oracle calls only allow for relatively few ($\sim 10^1$) iterations.

\subsection{Cox Regression}

We further test the same algorithms on the elastic-net regularized cox regression problem~\cite{10.1093/bioinformatics/btp322}. This formulation is widely used in survival analysis to model the relationship between a subject's genetic features and the time to an event (e.g., death or failure), while the non-smooth regularization promotes sparsity, allowing for the selection of relevant genes from high-dimensional data.

More specifically, given $ \{a_i \in \mathbb{R}^d, \delta_i \in \{0,1\}, t_i \in \mathbb{R}_+ \}_{i=1}^n $, we aim to model the association between an individual’s gene expression profile $a_i$their time-to-event outcome $t_{i}$. Here the censoring indicator $\sigma_{i}$ takes value $1$ if an event is observed and $0$ otherwise. Using the kidney renal clear cell carcinoma dataset\footnote{Available at http://gdac.broadinstitute.org/}, we can formulate this problem as
\[
  \min_{x\in \mathbb{R}^{d}}\, \frac{1}{n} \sum_{i=1}^n \delta_i \left\{ -a_i^T x + \log \left( \sum_{j \in \mathcal{R}_i} e^{a_j^T x} \right) \right\} + \frac{\mu}{2} \lVert x\rVert_{2}^{2} + \lambda \| x \|_1 .
\]
where \( x \in \mathbb{R}^m \) is the vector of covariates coefficients to be designed, $\mathcal{R}_i = \{ j : t_j \geq t_i \}$ is the set of subjects at risk at time \( t_i \). For this dataset, we have $n=112$ and $d=160$. As in previous experiments we also set $\lambda = \mu=10^{-4}$.

\begin{figure}
  \centering
  \includegraphics[width=.6\linewidth]{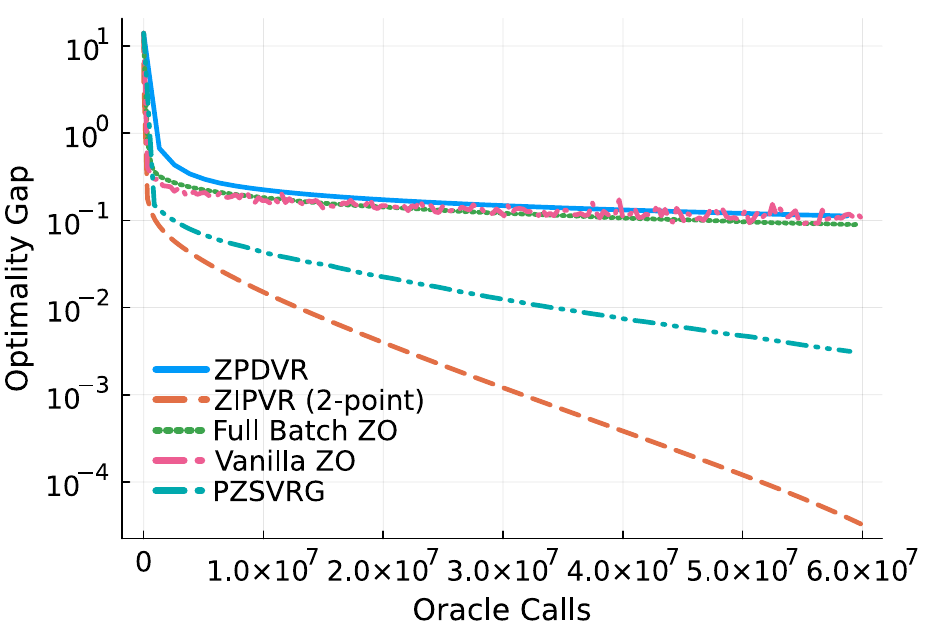}
  \caption{Comparison of different zeroth-order methods on the cox regression problem.}
    \label{fig:exp2}

\end{figure}

The result of our experiment is shown in Figure~\ref{fig:exp2}. Compared with the logistic-regression experiments, though the number of component function $n$ and problem dimension are smaller for the Cox regression problem, it exhibits a more challenging optimization landscape due to the nested log-sum-of-exponentials structure of the cost functions. Nevertheless, the comparative performance of the algorithms remains similar to the logistic-regression case. We can see that our proposed algorithm converges faster than other existing algorithms; among the existing methods, PZSVRG performs the best.

 \section{Conclusion}
In this work, we introduced \AlgAbbrv, a general variance reduction framework for zeroth-order composite optimization. By maintaining an incremental Jacobian estimator, \AlgAbbrv{} is able to overcome the issue of high maximum batch size of existing methods, enabling a practical pure 2-point implementation that eliminates the need for intermittent large-batch sampling. We established comprehensive convergence guarantees, showing that \AlgAbbrv{} matches the rates of first-order counterparts across strongly-convex, convex, and non-convex settings.

Given the flexibility of the \AlgAbbrv{} framework, several extensions and applications warrant further exploration.
First, it would be interesting to investigate accelerated variants of \AlgAbbrv{} for ill-conditioned strongly convex problems, aiming to improve the dependency of the oracle complexity on the condition number from $O\!\left(L/\mu\right)$ to $O\!\left(\sqrt{L/\mu}\right)$.
In addition, we are also interested in extending \AlgAbbrv{} to incorporate biased estimators to further enhance its performance in the non-convex setting, as algorithms using biased estimators (e.g., SPIDER~\cite{fang2018spider}, SARAH~\cite{nguyen2017sarah}) typically achieve superior rates for non-convex problems.
 \bibliographystyle{IEEEtran}
\bibliography{main_bib}
\appendix
\section{Derivation of \texorpdfstring{$\sigma$}{sigma} and Expected Batch Size $\nu$}

In this section, we derive the update probability parameter $\sigma$ and the expected batch size $\nu = \mathbb{E}[\omega_k |\mathcal{S}_k|]$ for the implementations presented in Section~\ref{section:algorithm}.

\begin{proof}
First, we note that, for the approach of constructing $\mathcal{S}_k$ proposed in Remark~\ref{rm:alg_sampling}, we have
\[
\begin{aligned}
\mathbb{E}\!\left[\omega_k P_k(A)\right]
={} & \mathbb{E}\Econd*{\sum_{(i,u)\in\mathcal{S}_k}uu^T A e_ie_i^T
}{\omega_k=1}\cdot\mathbb{P}(\omega_k=1) \\
={} & \mathbb{E}\Econd*{\sum_{i=1}^n\sum_{j=1}^d
\mathbb{1}_{(i,Q_k^{(j)})\in\mathcal{S}_k}
\cdot Q_k^{(j)}\big(Q_k^{(j)}\big)^T A e_ie_i^T
}{\omega_k=1}\cdot\mathbb{P}(\omega_k=1) \\
={} &
\sum_{i=1}^n\sum_{j=1}^d
\mathbb{P}\Vcond*{\big(i,Q_k^{(j)}\big)\in\mathcal{S}_k}{\omega_k=1}
\cdot Q_k^{(j)}\big(Q_k^{(j)}\big)^T A e_ie_i^T
\cdot\mathbb{P}(\omega_k=1).
\end{aligned}
\]
Then, it is not hard to check that for all the implementations presented in Section~\ref{section:algorithm}, the probability $\mathbb{P}\Vcond*{\big(i,Q_k^{(j)}\big)\in\mathcal{S}_k}{\omega_k=1}$ does not depend on $i,j$ or $k$, which we denote temporarily by $q$. As a result, we have
\[
\begin{aligned}
\mathbb{E}\!\left[\omega_k P_k(A)\right]
={} & q\cdot\mathbb{P}(\omega_k=1)\cdot \sum_{i=1}^n\sum_{j=1}^d
Q_k^{(j)}\big(Q_k^{(j)}\big)^T A e_ie_i^T
\cdot\mathbb{P}(\omega_k=1) \\
={} & q\cdot\mathbb{P}(\omega_k=1)\cdot A,
\end{aligned}
\]
showing that $\sigma=q\cdot\mathbb{P}(\omega_k=1)$. Furthermore, 
\[
\begin{aligned}
\nu= {} & \mathbb{E}[\omega_k|\mathcal{S}_k|]
=\mathbb{P}(\omega_k=1)\cdot\mathbb{E}\Econd*{|\mathcal{S}_k|}{\omega_k=1} \\
={} & \mathbb{P}(\omega_k=1)\cdot
\mathbb{E}\Econd*{\sum_{i=1}^n\sum_{j=1}^d
\mathbb{P}\Vcond*{\big(i,Q_k^{(j)}\big)\in\mathcal{S}_k}{\omega_k=1}}{\omega_k=1} \\
={} & \mathbb{P}(\omega_k=1)\cdot ndq
=nd\sigma,
\end{aligned}
\]
showing that $\nu$ is just $\sigma$ multiplied by $nd$.
We now analyze $\sigma$ of the four implementations separately.

\begin{enumerate}
    \item \textbf{Implementation I:} We have $\mathbb{P}(\omega_k=1) = 1$, and $\mathcal{S}_k$ consists of $R$ distinct pairs sampled uniformly from the $nd$ available options. Thus
    \[
    \sigma = q=\frac{\binom{nd-1}{R-1}}{\binom{nd}{R}} = \frac{R}{nd}.
    \]

    \item \textbf{Implementation II:} Let $p = \min\{R/d, 1\}$. The update occurs with probability $p$. If it occurs, we select $m = \lceil \frac{R}{pd} \rceil$ component indices uniformly from $[n]$ and update all $d$ directions for those components. The probability of a specific component $i$ being selected given $\omega_k=1$ is $m/n$, and thus
    \[
    \sigma = p \cdot q = p\cdot \frac{m}{n}
    =\min\{R/d,1\}\cdot\frac{\lceil R/(pd) \rceil}{n}=\begin{cases}
    R/(nd), & \text{if }R<d, \\
    n^{-1}\lceil R/d\rceil, & \text{if }R\geq d.
    \end{cases}.
    \]
    Note that if $R$ is not a multiple of $d$, we have $\sigma > R/nd$.

    \item \textbf{Implementation III:} We have $\mathbb{P}(\omega_k = 1) = R/nd$. When $\omega_k = 1$, the entire Jacobian is updated ($\mathcal{S}_k$ covers all $nd$ entries). Thus
    \[
    \sigma = \frac{R}{nd} \cdot 1 = \frac{R}{nd}.
    \]

    \item \textbf{Memory-Efficient Implementation:} We have $\mathbb{P}(\omega_k=1) = p = BR/nd$. If an update occurs, one block $\mathcal{T}_l$ is chosen uniformly from $B$ disjoint blocks partitioning $[n] \times [d]$. The probability a specific entry is in the chosen block is $1/B$. Thus
    \[
    \sigma = \frac{BR}{nd}\cdot q= \frac{BR}{nd}\cdot \frac{1}{B} = \frac{R}{nd}.
    \qedhere
    \]
\end{enumerate}
\end{proof}

\section{Proofs of Lemmas \ref{lm:projected_bias} and \ref{lm:estimator_bias}}
\label{appendix:auxiliary}

We will first introduce a lemma that bounds the error of the zeroth-order estimator $\widehat \nabla_{u} f_{i}(x ; \beta)$ in approximating the directional derivative $u u^{T} \nabla f_{i}(x)$
\begin{lemma}\label{lm:direc_bias}
  Suppose $u\in \mathbb{R}^{d}, \lVert u\rVert = 1$. Then for any $x\in \mathbb{R}^d$, $\beta > 0$ and $i \in [n]$ we have
	\begin{equation}\label{eq:two_point_bias}
		\left\|\widehat{\nabla}_u f_{i}(x ; \beta) - u u^{T} \nabla f_{i}(x)\right\|_{2} \leq \frac{  L\beta}{2}.
	\end{equation}
\end{lemma}
\begin{proof} We have
	\begin{align*}
		\left\|\widehat{\nabla}_u f_{i}(x ; \beta) - u u^{T} \nabla f_{i}(x)\right\|_{2} &= \left| \frac{f_{i}(x+\beta u) - f_{i}(x)}{\beta} - \langle \nabla f_{i}(x), u \rangle \right|\\
		&=\frac{1}{\beta}\left|\int_0 ^{\beta} \left(\langle\nabla f_{i}(x+\gamma u), u\rangle - \langle\nabla f_{i}(x), u\rangle\right) d\gamma \right|\\
		&\leq \frac{1}{\beta}\int_0 ^{\beta}  \|\nabla f_{i}(x+\gamma u)-\nabla f_{i}(x)\|_{2}\|u\|_{2}d\gamma\\
		&\leq \frac{1}{\beta} \int_{0}^{\beta} L\gamma \|u\|_{2}^2 d\gamma = \frac{L\beta}{2},
	\end{align*}
 where in the second step we used the fundamental theorem of calculus, and in the fourth step we used the $L$-smoothness of $f$.
\end{proof}

Next, we provide the proof of Lemma~\ref{lm:projected_bias}.

\begin{proof}[Proof of Lemma~\ref{lm:projected_bias}]
By the definitions of $G_{\beta_k}$ and $P_k$, we can write the squared Frobenius norm of the difference as:
\begin{align*}
    \left\lVert G_{\beta_{k}}(x_k; \mathcal{S}_{k}) - P_{k}(\nabla F(x_k)) \right\rVert^{2}_{F}
    ={} & \left\lVert \sum_{(i, u) \in \mathcal{S}_k} \left( \widehat{\nabla}_u f_i(x_k; \beta_k) - u u^T \nabla f_i(x_k) \right) e_i^T \right\rVert^{2}_{F}\\
    ={} & \sum_{i\in\mathcal{I}_k} \left\lVert \sum_{u\in\mathcal{U}_{k,i}}
    \left(\widehat{\nabla}_u f_i(x_k; \beta_k) - u u^T \nabla f_i(x_k) \right) \right\rVert^{2}_{2} \\
    ={} & \sum_{(i,u)\in\mathcal{S}_k} \left\lVert \widehat{\nabla}_u f_i(x_k; \beta_k) - u u^T \nabla f_i(x_k)  \right\rVert^{2}_{2},
\end{align*}
where we used the fact that the directions in $\mathcal{U}_{k,i}=\setv*{u}{(i, u) \in \mathcal{S}_k}$ are orthonormal any fixed $i$. Now Lemma~\ref{lm:projected_bias} directly follows by plugging \eqref{eq:two_point_bias} into the above inequality.
\end{proof}

Now we can present the proof of Lemma~\ref{lm:estimator_bias}.
\begin{proof}[Proof of Lemma~\ref{lm:estimator_bias}]
Let $\mathcal{D}$ denote either the uniform distribution over $\{e_1,\ldots,e_d\}$ or the uniform distribution over $\mathbb{S}^{d-1}$. Taking the conditional expectation of $g_k$ with respect to $\mathcal{F}_k$, and using the independence of indices and directions in $\mathcal{R}_k$, we have:
\begin{align*}
    \mathbb{E}\Econd*{g_{k}}{\mathcal{F}_{k}} 
    &= \frac{1}{n} J_k \mathbf{1} + \frac{d}{R}
    \cdot\frac{R}{n}\sum_{i=1}^n \left( \mathbb{E}_{u\sim\mathcal{D}}\Econd*{\widehat{\nabla}_u f_i(x_k; \beta_k)}{\mathcal{F}_k} - \mathbb{E}_{u\sim\mathcal{D}}\Econd*{u u^TJ_k^{(i)}}{\mathcal{F}_k} \right) \\
    &= \frac{1}{n} J_k \mathbf{1} + \frac{d}{n} \sum_{i=1}^n \mathbb{E}_{u\sim\mathcal{D}}\Econd*{\widehat{\nabla}_u f_i(x_k; \beta_k)}{\mathcal{F}_k} - \frac{1}{n} \sum_{i=1}^n   J_k^{(i)} \\
    &= \frac{d}{n} \sum_{i=1}^n \mathbb{E}_{u\sim\mathcal{D}}\Econd*{\widehat{\nabla}_u f_i(x_k; \beta_k)}{\mathcal{F}_k},
\end{align*}
where we used $\mathbb{E}_{u\sim\mathcal{D}}[uu^T]=d^{-1}I_d$.

Now we discuss the two choices of $\mathcal{D}$ separately.

\paragraph{Random Coordinate Direction:} If $\mathcal{D}$ is the uniform distribution over $\{ e_{1}, e_{2}, \ldots, e_{d}\}$, then
\begin{align*}
  \left\lVert \mathbb{E}\Econd*{g_{k}}{\mathcal{F}_{k}} - \nabla f(x_{k}) \right\rVert^{2} 
  &= \left\lVert \frac{1}{n} \sum_{j=1}^{d} \sum_{i=1}^{n} \left(\widehat \nabla_{e_{j}} f_{i}(x_{k};\beta_k) - e_{j}e_{j}^{T} \nabla f_{i}(x_{k})\right)\right\rVert^{2}\\
  &= \sum_{j=1}^{d} \left\lVert \frac{1}{n}  \sum_{i=1}^{n} \left(\widehat \nabla_{e_{j}} f_{i}(x_{k};\beta_k) - e_{j}e_{j}^{T} \nabla f_{i}(x_{k})\right)\right\rVert^{2}\\
  &= \sum_{j=1}^{d} \left\lVert \widehat \nabla_{e_{j}} f(x_{k};\beta_k) - e_{j}e_{j}^{T} \nabla f(x_{k})\right\rVert^{2} 
  \leq \frac{d L^{2} \beta_k^{2}}{4},
\end{align*}
where in the inequality we use \eqref{eq:two_point_bias} on the $L$-smooth function $f$.

\paragraph{Random Direction on the Unit Sphere:} If $\mathcal{D}$ is the uniform distribution over the unit sphere $\mathbb{S}^{d-1}$, then we can directly apply \cite[Lemma 14]{malik2020derivative} to get
\begin{equation*}
\begin{aligned}
  \left\lVert \mathbb{E}\Econd*{g_{k}}{\mathcal{F}_{k}} - \nabla f(x_{k}) \right\rVert^{2} ={} & \left\lVert \mathbb{E}_{u\sim\mathcal{D}}\Econd*{ d\widehat \nabla_{u} f(x_{k};\beta_k)}{\mathcal{F}_{k}} - \nabla f(x_{k})\right\rVert^{2} \\
  \leq{} & \frac{L^{2} \beta_k^{2}}{4} \leq \frac{d L^{2} \beta_k^{2}}{4},
\end{aligned}
\end{equation*}
where the first inequality follows from \cite[Lemma 14]{malik2020derivative}.

Putting the two cases together, we arrive at the desired bound.
\end{proof}

\section{Proofs for the Strongly Convex Case}
\label{appendix:proof_strongly_convex}

\subsection{Proof of Lemma~\ref{lm:sc_gradient_bound}}
From Young's inequality, we have
\begin{equation*}
    \| g_k - \nabla f(x^\ast)\|_{2}^2 
    \leq  2 \|\nabla f(x_k) - \nabla f(x^\ast)\|_{2}^2 + 2\|g_k - \nabla f(x_k)\|_{2}^2.
\end{equation*}
For the first term, we notice that
\begin{align*}
    2 \|\nabla f(x_k) - \nabla f(x^\ast)\|_{2}^2 
    ={} &
    2\left\| \frac{1}{n} \sum_{i=1}^{n} (\nabla f_{i}(x_{k}) - \nabla f_{i}(x^{\ast}))\right\|_{2}^2
    \\ \leq{}  &
    \frac{2}{n} \sum_{i=1}^{n} \left\| \nabla f_{i}(x_{k}) - \nabla f_{i}(x^{\ast})\right\|_{2}^2
    \\ ={}  &
    \frac{2}{n}\left\| \nabla F(x_{k}) - \nabla F(x^{\ast})\right\|_{F}^2,
\end{align*}
where the inequality comes from the Cauchy-Schwarz inequality.
For the second term, we start from the bound provided by~\eqref{eq:nc_grad_var_bound}:
\begin{align*}
  2\mathbb{E}\Econd*{\|g_k - \nabla f(x_k)\|_{2}^2}{\mathcal{F}_{k}}
    \leq{} &
     \frac{4d}{nR} \| \nabla F(x_k) - J_k\|_F^2
   + d^2 L^2 \beta_k^2
    \\ \leq{} &
    \frac{8d}{n R} \| \nabla F(x_k) - \nabla F(x^{\ast})\|_F^2
     +\frac{8d}{nR} \| J_{k} - \nabla F(x^{\ast})\|_F^2 
   + d^2 L^2 \beta_k^2.
\end{align*}
Combining the two inequalities above, we finish the proof of~\eqref{eq:sc_gradient_bound}.

\subsection{Proof of Lemma~\ref{lm:sc_jacobian_recur}}
To prove Lemma~\ref{lm:sc_jacobian_recur}, we will need the following auxilliary lemma:
\begin{lemma}
\label{lm:op_twice}
Under Assumption~\ref{as:algorithm_rand_var}, for any matrix $A \in \mathbb{R}^{d \times n}$, the operator $P_k$ satisfies:
\begin{equation}\label{eq:app_op_twice}
    \|P_k(A)\|_F^2 = \sum_{(i, u) \in \mathcal{S}_k} (u^T A e_i)^2 = \langle P_k(A), A \rangle_F.
\end{equation}
\end{lemma}

\begin{proof}
  Recall the definition $P_k(A) = \sum_{(i, u) \in \mathcal{S}_k} u u^T A e_i e_i^T$. We have
  \begin{align*}
  \left\|P_k(A)\right\|_F^2
  ={} & \tr\!\Bigg[
  \sum_{(i, u) \in \mathcal{S}_k} \!\!\!u u^T A e_i e_i^T
  \bigg\{\!\sum_{(j, v) \in \mathcal{S}_k} \!\!\!v v^T A e_j e_j^T\bigg\}^{\!T}
  \Bigg] \\
  ={} & \sum_{(i, u) \in \mathcal{S}_k}\sum_{(j, v) \in \mathcal{S}_k}
  \tr\!\left(u^T A e_i e_i^Te_je_j^TA^T
  v v^T u\right) \\
  ={} & \sum_{(i, u) \in \mathcal{S}_k}\sum_{v:(i, v) \in \mathcal{S}_k}\tr\!\left(u^T A e_i e_i^TA^T
  v v^T u\right) \\
  ={} & \sum_{(i, u) \in \mathcal{S}_k}\tr\!\left(u^T A e_i e_i^TA^T
  u\right)
  =\sum_{(i,u)\in\mathcal{S}_k} (u^TAe_i)^2.
  \end{align*}
  Here the third step follows by noting that $e_i^Te_j$ equals $1$ if $i=j$ and equals $0$ if $i\neq j$; the fourth step follows by the requirement in Assumption~\ref{as:algorithm_rand_var} that the vectors in $\setv*{u}{(i, u) \in \mathcal{S}_k}$ are orthonormal.

Similarly, for the inner product, we have
\begin{align*}
    \langle P_k(A), A \rangle_F 
    &= \sum_{(i, u) \in \mathcal{S}_k} \tr(e_i e_i^T A^T u u^T A)
    = \sum_{(i, u) \in \mathcal{S}_k} (u^T A e_i)^2.
\end{align*}
Comparing the two results completes the proof.
\end{proof}

Now we provide a proof for Lemma \ref{lm:sc_jacobian_recur}.

\begin{proof}[Proof of Lemma \ref{lm:sc_jacobian_recur}]
Let $A_k = J_k - \nabla F(x^*)$ and $B_k = \nabla F(x_k) - \nabla F(x^*)$. The update rule can be rewritten as:
\[
    J_{k+1} - \nabla F(x^*) = \underbrace{A_k - \omega_k P_k(A_k - B_k)}_{Y_k} + \underbrace{\omega_k (G_{\beta_k}(x_k; \mathcal{S}_k) - P_k(\nabla F(x_k)))}_{Z_k}.
\]
Using Young's inequality,  we have $\|Y_k + Z_k\|_F^2 \leq (1 + \frac{\sigma}{2}) \|Y_k\|_F^2 + (1 + \frac{2}{\sigma}) \|Z_k\|_F^2$.

First, we bound $\mathbb{E}\Econd*{\|Z_k\|_F^2}{\mathcal{F}_k}$. Since $\omega_k^2 = \omega_k$, applying  Lemma \ref{lm:projected_bias} leads to
\[
    \mathbb{E}\Econd*{\|Z_k\|_F^2}{\mathcal{F}_k} = \mathbb{E}\Econd*{\omega_k \|G_{\beta_k}(x_k; \mathcal{S}_k) - P_k(\nabla F(x_k))\|_F^2}{\mathcal{F}_k} \leq \frac{L^2 \beta_k^2}{4} \mathbb{E}\Econd*{\omega_k |\mathcal{S}_k|}{\mathcal{F}_k} = \frac{\nu L^2 \beta_k^2}{4}.
\]

 Next, we bound $\mathbb{E}\Econd*{\|Y_k\|_F^2}{\mathcal{F}_k}$. By expanding the squared norm, we have
\[
    \|Y_k\|_F^2 = \|A_k\|_F^2 + \|\omega_k P_k(A_k - B_k)\|_F^2 - 2 \langle A_k, \omega_k P_k(A_k - B_k) \rangle_F.
\]
Now we apply~\eqref{eq:app_op_twice} to obtain
\begin{align*}
    \mathbb{E}\Econd*{\|\omega_k P_k(A_k - B_k)\|_F^2}{\mathcal{F}_k} 
    &= \mathbb{E}\Econd*{\omega_k \langle P_k(A_k - B_k), A_k - B_k \rangle_F}{\mathcal{F}_k} \\
    &= \langle \sigma(A_k - B_k), A_k - B_k \rangle_F = \sigma \|A_k - B_k\|_F^2.
\end{align*}
Also note that the expectation of the cross term $- 2 \langle A_k, \omega_k P_k(A_k - B_k) \rangle_F$ is $-2 \langle A_k, \sigma(A_k - B_k) \rangle_F$. Combining these, we have
\begin{align*}
    \mathbb{E}\Econd*{\|Y_k\|_F^2}{\mathcal{F}_k} 
    &= \|A_k\|_F^2 + \sigma \|A_k - B_k\|_F^2 - 2\sigma \langle A_k, A_k - B_k \rangle_F \\
    &= (1 - \sigma) \|A_k\|_F^2 + \sigma \|B_k\|_F^2.
\end{align*}
Finally, combining the bounds for $Y_k$ and $Z_k$, we get
\begin{align*}
    \mathbb{E}\Econd*{\|J_{k+1} - \nabla F(x^*)\|_F^2}{\mathcal{F}_k} 
    &\leq \left(1 + \frac{\sigma}{2}\right) \left( (1 - \sigma) \|A_k\|_F^2 + \sigma \|B_k\|_F^2 \right) + \left(1 + \frac{2}{\sigma}\right) \frac{\nu L^2 \beta_k^2}{4} \\
    &\leq \left(1 - \frac{\sigma}{2} - \frac{\sigma^2}{2}\right) \|A_k\|_F^2 + \left(\sigma + \frac{\sigma^2}{2}\right) \|B_k\|_F^2 + \left(1 + \frac{2}{\sigma}\right) \frac{\nu L^2 \beta_k^2}{4}.
\end{align*}
Relaxing the coefficients using $\sigma \in (0, 1)$ (specifically $1 - \frac{\sigma}{2} - \frac{\sigma^2}{2} \leq 1 - \frac{\sigma}{2}$ and $\sigma + \frac{\sigma^2}{2} \leq 2\sigma$) yields the stated result.
\end{proof}

\section{Proof of Lemma~\ref{lm:c_alt_distance}}
\label{sec:app_proof_lm_c_alt_distance}

Before proving Lemma~\ref{lm:c_alt_distance}, we introduce the following lemma about the non-expansiveness of a proximal gradient step with true gradient. 
\begin{lemma}
   Suppose Assumption 1 holds and each $f_i$ is convex. For any $\alpha \in (0, 2/L]$, the mapping $T(x) = \operatorname{prox}_{\alpha \psi}(x - \alpha \nabla f(x))$ is a non-expansion, i.e., for any $x, y \in \mathbb{R}^d$:
\begin{equation}\label{eq:nc_cmp_nonexp}
    \|T(x) - T(y)\|_2 \leq \|x - y\|_2.
\end{equation}
In particular, since $x^*$ is a fixed point of $T$, it holds that $\|\text{prox}_{\alpha \psi}(x - \alpha \nabla f(x)) - x^*\|_2 \leq \|x - x^*\|_2$. 
\end{lemma}

\begin{proof}
Since the proximal operator $\operatorname{prox}_{\alpha \psi}(\cdot)$ is non-expansive, we have
\begin{equation} \label{eq:nc_prox_nonexp}
    \|T(x) - T(y)\|_2 \leq \|(x - \alpha \nabla f(x)) - (y - \alpha \nabla f(y))\|_2.
\end{equation}
Expanding the squared norm on the right-hand side of \eqref{eq:nc_prox_nonexp} gives
\begin{equation*}
\begin{aligned}
\|(x - y) - \alpha(\nabla f(x) - \nabla f(y))\|_2^2 ={} & \|x - y\|_2^2 - 2\alpha \langle \nabla f(x) - \nabla f(y), x - y \rangle + \alpha^2 \|\nabla f(x) - \nabla f(y)\|_2^2 \\
\leq{} & \|x - y\|_2^2 - \frac{2\alpha}{L} \|\nabla f(x) - \nabla f(y)\|_2^2 + \alpha^2 \|\nabla f(x) - \nabla f(y)\|_2^2 \\
={} &\|x - y\|_2^2 - \alpha \left( \frac{2}{L} - \alpha \right) \|\nabla f(x) - \nabla f(y)\|_2^2,
\end{aligned}
\end{equation*}

Since $\alpha\in(0,2/L]$ implies $\alpha(2/L - \alpha)\leq 0$, the last term can further be dropped, which, combined with~\eqref{eq:nc_prox_nonexp}, leads to~\eqref{eq:nc_cmp_nonexp}.
\end{proof}

We are now ready to prove Lemma~\ref{lm:c_alt_distance}.

\begin{proof}[Proof of Lemma \ref{lm:c_alt_distance}]
Recall that $x_{k+1} = \text{prox}_{\alpha \psi}(x_k - \alpha g_k)$. By the optimality condition of the proximal operator, there exists a subgradient $\xi \in \partial \psi(x_{k+1})$ such that
\begin{equation} \label{eq:prox_opt}
    \xi = \frac{1}{\alpha}(x_k - x_{k+1}) - g_k.
\end{equation}
By the convexity of $\psi$, for the optimal solution $x^*$, we have:
\begin{equation} \label{eq:psi_conv}
    \psi(x^*) \geq \psi(x_{k+1}) + \langle \xi, x^* - x_{k+1} \rangle.
\end{equation}
Next, we treat the smooth part $f$. By the convexity and $L$-smoothness of $f$, we have
\begin{gather*}
    f(x^*) \geq f(x_k) + \langle \nabla f(x_k), x^* - x_k \rangle, \label{eq:f_conv} \\
    f(x_k) \geq f(x_{k+1}) - \langle \nabla f(x_k), x_{k+1} - x_k \rangle - \frac{L}{2} \|x_{k+1} - x_k\|_2^2. \label{eq:f_smooth}
\end{gather*}
Summing the above two inequalities yields:
\begin{equation} \label{eq:f_bound}
    f(x^*) \geq f(x_{k+1}) + \langle \nabla f(x_k), x^* - x_{k+1} \rangle - \frac{L}{2} \|x_{k+1} - x_k\|_2^2.
\end{equation}
Now, combining \eqref{eq:psi_conv} and \eqref{eq:f_bound}, and using $h = f + \psi$, we obtain
\begin{equation*}
    h(x^*) \geq h(x_{k+1}) + \langle \nabla f(x_k) + \xi, x^* - x_{k+1} \rangle - \frac{L}{2} \|x_{k+1} - x_k\|_2^2.
\end{equation*}
Substituting $\xi$ from \eqref{eq:prox_opt} and assuming $\alpha \leq 1/L$ gives
\begin{equation*} 
    h(x^*) \geq h(x_{k+1}) + \frac{1}{\alpha} \langle x_k - x_{k+1}, x^* - x_{k+1} \rangle + \langle \nabla f(x_k) - g_k, x^* - x_{k+1} \rangle - \frac{1}{2\alpha} \|x_{k+1} - x_k\|_2^2.
\end{equation*}
Using the identity $\|x_{k+1} - x^*\|_2^2 = \|x_k - x^*\|_2^2 - \|x_{k+1} - x_k\|_2^2 + 2\langle x_{k+1} - x_k, x_{k+1} - x^* \rangle$, then taking expectation conditioned on $\mathcal{F}_k$, the above inequality can be rearranged into a recursive bound on the distance to $x^*$:
\begin{align}
        \mathbb{E}\Econd*{\|x_{k+1} - x^*\|_2^2}{\mathcal{F}_k} \leq& \|x_k - x^*\|_2^2 - 2\alpha \mathbb{E}\Econd*{(h(x_{k+1}) - h(x^*))}{\mathcal{F}_k} 
        \nonumber \\& 
        + 2\alpha \mathbb{E}\Econd*{\langle g_k - \nabla f(x_k), x^* - x_{k+1} \rangle}{\mathcal{F}_k}.\label{eq:nc_dist_recur}
\end{align} 
We now handle the stochastic cross term. Let $\overline x_{k+1} = \text{prox}_{\alpha \psi}(x_k - \alpha \nabla f(x_k))$ be the point reached by a deterministic proximal gradient step. Then we have
\begin{align}
    \mathbb{E}\Econd*{\langle g_k - \nabla f(x_k), x^* - x_{k+1} \rangle}{\mathcal{F}_k} ={} & \langle \mathbb{E}\Econd*{g_k - \nabla f(x_k)}{\mathcal{F}_k}, x^* - \overline x_{k+1} \rangle \nonumber \\
    & + \mathbb{E}\Econd*{\langle g_k - \nabla f(x_k), \overline x_{k+1} - x_{k+1} \rangle}{\mathcal{F}_k}. \label{eq:nc_cross_decomp}
\end{align}
By using Cauchy-Schwarz inequality on the first term of \eqref{eq:nc_cross_decomp}, we get
\begin{align*}
\langle \mathbb{E}\Econd*{g_k - \nabla f(x_k)}{\mathcal{F}_k}, x^* - \overline x_{k+1} \rangle
\leq{} & \lVert\mathbb{E}\Econd*{g_k - \nabla f(x_k)}{\mathcal{F}_k}\rVert_2\cdot
\lVert x^* - \overline x_{k+1} \rVert_2 \\
\leq{} & \frac{\beta_kL\sqrt{d}}{2}\lVert T(x^\ast)-T(x_k)\rVert_2
\leq \frac{\beta_kL\sqrt{d}}{2}\lVert x^\ast-x_k\rVert_2 \\
\leq{} &\frac{\beta_kL}{2C_2}\|x^\ast-x_k\|_2^2 + \frac{\beta_kdLC_2}{8},
\end{align*}

where the second inequality comes from \eqref{eq:sc_estimator_bias} and the third comes from \eqref{eq:nc_cmp_nonexp}. 
To bound the second term of~\ref{eq:nc_cross_decomp}, note that the non-expansivity of the proximal operator implies $\|x_{k+1} - \overline x_{k+1}\|_{2} \leq \alpha \|g_k - \nabla f(x_k)\|_{2}$. We then apply the Cauchy-Schwarz inequality on the second term of \eqref{eq:nc_cross_decomp} and use the bound~\eqref{eq:nc_grad_var_bound} to get
\begin{align*}
    &\mathbb{E}\Econd*{\langle g_k - \nabla f(x_k), \overline x_{k+1} - x_{k+1} \rangle}{\mathcal{F}_k} \\
    \leq{} & \alpha \mathbb{E}\Econd*{\|g_k - \nabla f(x_k)\|_2^2}{\mathcal{F}_k} \\
    \leq{} &
    \alpha\left(\frac{2d}{nR} \| \nabla F(x_k) - J_k\|_F^2
   + \frac{d^2 L^2 \beta_k^2}{2}\right) \\
    \leq{} & \frac{4 \alpha d}{n R} \| \nabla F(x_k) - \nabla F(x^{\ast})\|_F^2 + \frac{4 \alpha d}{nR } \| J_{k} - \nabla F(x^{\ast})\|_F^2 
   + \frac{\alpha d^2 L^2 \beta_k^2}{2},
\end{align*}
where the last step follows from Young's inequality.
Plugging these two bounds back to \eqref{eq:nc_cross_decomp}, we have 
\begin{align*}
        \mathbb{E}\Econd*{\langle g_k - \nabla f(x_k), x^* - x_{k+1} \rangle}{\mathcal{F}_k} &\leq 
        \frac{4 \alpha d}{n R} \| \nabla F(x_k) - \nabla F(x^{\ast})\|_F^2 + \frac{4 \alpha d}{nR } \| J_{k} - \nabla F(x^{\ast})\|_F^2 
       \\ &\quad + 
       \frac{\beta_kL}{2C_2} \|x^* -  x_{k}\|_2^2 + \frac{\beta_k d LC_2}{8}+ \frac{\alpha d^2 L^2 \beta_k^2}{2} .
\end{align*}
Plugging this bound to \eqref{eq:nc_dist_recur} finishes the proof.
\end{proof}

\section{Proofs for the Non-convex Case}
\label{app:proof_nonconvex}

\subsection{Proof of Lemma~\ref{lm:nc_grad_var_bound}}
\begin{proof}
We know that
\begin{align*}
    g_{k} - \nabla f(x_{k})
    &={} 
    \frac{d}{R}\sum_{(i, u) \in\mathcal{R}_k}
  \left(u u^T \nabla f_i(x_k)
  -uu^T J_k^{(i)}\right)
  -\left(\nabla f(x_k)- \frac{1}{n}J_k\cdot\mathbf{1}\right)\\
  &\quad +\frac{d}{R}\sum_{(i, u) \in\mathcal{R}_k}
    \left(\widehat \nabla_u f_i(x_k;\beta_k) - u u^T \nabla f_i(x_k)\right).
\end{align*}
By Young's inequality and Lemma \ref{lm:direc_bias}, we have
  \begin{align*}
      \lVert
      g_{k} - \nabla f(x_{k})
    \rVert_{2}^{2}
      \leq{} & 2 \left\lVert
    \frac{d}{R}\sum_{(i, u) \in\mathcal{R}_k}
   u u^T  \left( \nabla f_i(x_k)
  - J_k^{(i)}\right)
  -\left(\nabla f(x_k)- \frac{1}{n}J_k\cdot\mathbf{1}\right)
    \right\rVert_{2}^{2}
    + \frac{d^2 L^2 \beta_k^2}{2} \\
    ={} &
    2 \left\lVert
    \frac{1}{R}\sum_{r=1}^R
    \left[
   du_k^r (u_k^r)^T  \left( \nabla f_{i_k^r}(x_k)
  - J_k^{(i_k^r)}\right)
  -\left(\nabla f(x_k)- \frac{1}{n}J_k\cdot\mathbf{1}\right)\right]
    \right\rVert_{2}^{2}
    + \frac{d^2 L^2 \beta_k^2}{2}
    \numberthis \label{eq:nc_gvar_1}.
  \end{align*}
  Note that for each $r=1,\ldots,R$, we have
  \[
  \begin{aligned}
  \mathbb{E}\Econd*{ 
  du_k^r(u_k^r)^T\left( \nabla f_{i_k^r}(x_k) - J_k^{(i_k^r)}\right)}{\mathcal{F}_k}
  ={} &
  \mathbb{E}\Econd*{ 
  du_k^r(u_k^r)^T}{\mathcal{F}_k}\mathbb{E}\Econd*{\nabla f_{i_k^r}(x_k) - J_k^{(i_k^r)}}{\mathcal{F}_k} \\
  ={} & \frac{1}{n}\sum_{i=1}^n\nabla f_{i}(x_k) - J_k^{(i)}
  = \nabla f(x_k)- \frac{1}{n}J_k\cdot\mathbf{1},
  \end{aligned}
  \]
  where the first step follows by the (conditional) independence of $u_k^r$ and $i_k^r$, and the remaining steps follows by how each $u_k^r$ and $i^r_k$ are generated. Furthermore,
  \[
  \begin{aligned}
  & \mathbb{E}\Econd*{ 
  \left\|du_k^r(u_k^r)^T\left( \nabla f_{i_k^r}(x_k) - J_k^{(i_k^r)}\right)\right\|_{2}^2}{\mathcal{F}_k} \\
  ={} & d^2 \tr\mathbb{E}\Econd*{ 
  \left( \nabla f_{i_k^r}(x_k) - J_k^{(i_k^r)}\right)\left( \nabla f_{i_k^r}(x_k) - J_k^{(i_k^r)}\right)^Tu_k^r(u_k^r)^Tu_k^r(u_k^r)^T}{\mathcal{F}_k} \\
  ={} & d^2\tr\!\left\{\mathbb{E}\Econd*{ 
  \left( \nabla f_{i_k^r}(x_k) - J_k^{(i_k^r)}\right)\left( \nabla f_{i_k^r}(x_k) - J_k^{(i_k^r)}\right)^T}{\mathcal{F}_k}\mathbb{E}\Econd*{u_k^r(u_k^r)^T}{\mathcal{F}_k}\right\} \\
  ={} & d\tr\left\{
  \frac{1}{n}\sum_{i=1}^n
  \left( \nabla f_{i}(x_k) - J_k^{(i)}\right)\left( \nabla f_{i}(x_k) - J_k^{(i)}\right)^T
  \right\} \\
  ={} & \frac{d}{n}\|\nabla F(x_k)-J_k\|_F^2,
  \end{aligned}
  \]
  where we used the (conditional) independence of $u_k^r$ and $i_k^r$ and $\mathbb{E}\Econd*{u_k^r(u_k^r)^T}{\mathcal{F}_k}=d^{-1}I_d$.
  Thus
    \begin{align*}
       &\mathbb{E}\Econd*{\lVert
      g_{k} - \nabla f(x_{k})
    \rVert_{2}^{2}}{\mathcal{F}_k}\\
             \leq{} & \frac{2}{R^2}
             \sum_{r=1}^R\mathbb{E}\Econd*{ 
  \left\|du_k^r(u_k^r)^T\left( \nabla f_{i_k^r}(x_k) - J_k^{(i_k^r)}\right)
  -\left(\nabla f(x_k)- \frac{1}{n}J_k\cdot\mathbf{1}\right)\right\|^2}{\mathcal{F}_k}
  + \frac{d^2 L^2 \beta_k^2}{2}\\
  \leq{} & \frac{2}{R^2}
             \sum_{r=1}^R\mathbb{E}\Econd*{ 
  \left\|du_k^r(u_k^r)^T\left( \nabla f_{i_k^r}(x_k) - J_k^{(i_k^r)}\right)\right\|_{2}^2}{\mathcal{F}_k}
  + \frac{d^2 L^2 \beta_k^2}{2} \\
  \leq{} & \frac{2d}{nR} \| \nabla F(x_k) - J_k\|_F^2
   + \frac{d^2 L^2 \beta_k^2}{2}.
  \end{align*}
The proof is now complete.
\end{proof}

\subsection{Proof of Lemma~\ref{lm:nc_jacobian_error}}

\begin{proof}
We decompose the error into a drift term and an update error term:
\[
    \nabla F(x_{k+1}) - J_{k+1} = (\nabla F(x_{k+1}) - \nabla F(x_k)) + (\nabla F(x_k) - J_{k+1}).
\]
Applying Young's inequality leads to
\begin{equation}
\label{eq:grad_recursion_split}
    \lVert \nabla F(x_{k+1}) - J_{k+1} \rVert_{F}^{2} 
    \leq  \left(1 + \frac{\sigma}{4}\right) \lVert J_{k+1} - \nabla F(x_k) \rVert_{F}^{2} 
    + \left(1 + \frac{4}{\sigma}\right) \lVert \nabla F(x_{k+1}) - \nabla F(x_k) \rVert_{F}^{2} 
\end{equation}
We further treat the first term of~\eqref{eq:grad_recursion_split} as follows:
\begin{align*}
  \lVert J_{k+1} - \nabla F(x_k) \rVert_{F}^{2}
  & = \lVert (I - \omega_k P_k)(J_k - \nabla F(x_k)) + \omega_k (G_{\beta_k}(x_k; \mathcal{S}_k) - P_k(\nabla F(x_k)))\rVert_{F}^{2}\\
  &\leq \left(1 + \frac{\sigma}{4}\right) \lVert (I - \omega_k P_k)(J_k - \nabla F(x_k)) \rVert_{F}^{2} + \left(1 + \frac{4}{\sigma}\right) \frac{ \omega_{k} \lvert \mathcal{S}_{k}\rvert L^2 \beta_k^2}{4}\numberthis \label{eq:nc_jac_2},
\end{align*}
where the inequality comes from Young's inequality and Lemma \ref{lm:projected_bias}. For the first term of \eqref{eq:nc_jac_2}, we first expand the squared Frobenius norm:
\begin{align*}
    \lVert (I - \omega_k P_k)(J_k - \nabla F(x_k)) \rVert_{F}^{2}
    ={} & \lVert J_k - \nabla F(x_k) \rVert_F^2 - 2 \langle J_k - \nabla F(x_k), \omega_k P_k(J_k - \nabla F(x_k)) \rangle_F \\
    & + \lVert \omega_k P_k(J_k - \nabla F(x_k)) \rVert_F^2.
\end{align*}
By $\omega_k^2 = \omega_k$ and~\eqref{eq:app_op_twice}, we get
\begin{align*}
    \lVert (I - \omega_k P_k)(J_k - \nabla F(x_k)) \rVert_{F}^{2}
    ={} & \lVert J_k - \nabla F(x_k) \rVert_F^2 - 2\omega_k \langle P_k(J_k - \nabla F(x_k)), J_k - \nabla F(x_k) \rangle_F \\
    & + \omega_k \langle P_k(J_k - \nabla F(x_k)), J_k - \nabla F(x_k) \rangle_F \\
    ={} & \lVert J_k - \nabla F(x_k) \rVert_F^2 - \omega_k \langle P_k(J_k - \nabla F(x_k)), J_k - \nabla F(x_k) \rangle_F.
\end{align*}
Now, we take the conditional expectation with respect to $\mathcal{F}_k$. Using our requirement~\eqref{eq:alg_omega_req1}, we have
\begin{align*}
    & \mathbb{E}\Econd*{\lVert (I - \omega_k P_k)(J_k - \nabla F(x_k)) \rVert_{F}^{2}}{\mathcal{F}_k} \\
    ={} & \lVert J_k - \nabla F(x_k) \rVert_F^2 - \langle \mathbb{E}\Econd*{\omega_k P_k(J_k - \nabla F(x_k))}{\mathcal{F}_k}, J_k - \nabla F(x_k) \rangle_F \\
    ={} & \lVert J_k - \nabla F(x_k) \rVert_F^2 - \langle \sigma (J_k - \nabla F(x_k)), J_k - \nabla F(x_k) \rangle_F \\
    ={} & (1 - \sigma) \lVert J_k - \nabla F(x_k) \rVert_F^2.
\end{align*}
Combining this with the definition of $\nu$ and taking expectations on \eqref{eq:nc_jac_2} gives
\begin{equation*}
  \mathbb{E} \Econd*{\lVert J_{k+1} - \nabla F(x_k) \rVert_{F}^{2}}{\mathcal{F}_{k}}
   \leq \left(1+ \frac{\sigma}{4}\right)(1 - \sigma) \lVert J_k - \nabla F(x_k) \rVert_F^2
   + \left(1 + \frac{4}{\sigma}\right) \frac{ \nu L^2 \beta_k^2}{4}.
\end{equation*}
Now we can plug the above bound into~\eqref{eq:grad_recursion_split} and take conditional expectations with respect to $\mathcal{F}_{k}$. Noting that $(1+\frac{\sigma}{4})^2(1-\sigma) \leq 1 - \frac{\sigma}{2}$ and $(1+\frac{\sigma}{4})(1+\frac{4}{\sigma}) \leq \frac{10}{\sigma}$ for $\sigma \in (0, 1)$, we obtain
\begin{align}
  \label{eq:nc4acc}
  &\mathbb{E}\Econd*{\lVert \nabla F(x_{k+1}) - J_{k+1} \rVert_{F}^{2}}{\mathcal{F}_{k}}
  \nonumber \\ \leq{} &
    \left(1- \frac{\sigma}{2}\right) \lVert \nabla F(x_{k}) - J_{k})\rVert_{F}^{2}
      + \frac{5  }{\sigma} \mathbb{E}\Econd*{\lVert \nabla F(x_{k+1}) - \nabla F(x_{k}) \rVert_{2}^{2}}{\mathcal{F}_k} 
      +  \frac{5 \nu n d L^{2} \beta_{k}^{2} }{2 \sigma}
   \\ \leq{} &
    \left(1- \frac{\sigma}{2}\right) \lVert \nabla F(x_{k}) - J_{k})\rVert_{F}^{2}
      + \frac{5n L^{2} }{\sigma} \mathbb{E}\Econd*{\lVert x_{k+1} - x_{k}\rVert_{2}^{2}}{\mathcal{F}_k} 
      +  \frac{5 \nu n d L^{2} \beta_{k}^{2} }{2 \sigma},\nonumber 
\end{align}
where in the last step we use the $L$-smoothness of the component functions.
\end{proof}

\subsection{Proof of Lemma \ref{lm:nc_frac_g}}

\begin{proof}
Using Lemma 2 in~\cite{j2016proximal}, for any $y \in \mathbb{R}^{d}$, we have
  \begin{align}
    \label{eq:nc_frac_g1}
    h(x_{k+1}) \leq{} &
      h(y) + \langle x_{k+1} -y, \nabla f(x) - g_{k}\rangle 
      + \left(\frac{L}{2} - \frac{1}{2 \alpha}\right) \lVert x_{k+1} - x_{k}\rVert_{2}^{2} 
             \nonumber \\ &
      + \left( \frac{L}{2} + \frac{1}{2 \alpha}\right) \lVert y - x_{k}\rVert_{2}^{2}
      - \frac{1}{2 \alpha} \lVert x_{k+1} - y\rVert_{2}^{2}.
  \end{align}
  Now, let $\overline x_{k+1} = \text{prox}_{\alpha \psi}(x_k - \alpha \nabla f(x_k))$ denote the point reached by a deterministic proximal gradient step. Applying Lemma 2 in~\cite{j2016proximal} again implies that for any $y \in \mathbb{R}^{d}$:
  \begin{equation}
    \label{eq:nc_frac_g2}
    h(\overline x_{k+1}) \leq
      h(y)  + \left(\frac{L}{2} - \frac{1}{2 \alpha}\right) \lVert \overline x_{k+1} - x_{k}\rVert_{2}^{2} 
      + \left( \frac{L}{2} + \frac{1}{2 \alpha}\right) \lVert y - x_{k}\rVert_{2}^{2}
      - \frac{1}{2 \alpha} \lVert \overline x_{k+1} - y\rVert_{2}^{2}.
  \end{equation}
  We substitute $y= \overline x_{k+1}$ into \eqref{eq:nc_frac_g1} and $y= x_{k}$ into \eqref{eq:nc_frac_g2}. Summing the resulting inequalities and applying Young's inequality to the inner product term yields
  \begin{align*}
    h(x_{k+1}) \leq{} &
      h(x_{k}) + \langle x_{k+1} -\overline x_{k+1}, \nabla f(x) - g_{k}\rangle 
      + \left(\frac{L}{2} - \frac{1}{2 \alpha}\right) \lVert x_{k+1} - x_{k}\rVert_{2}^{2} 
             \\&
             + \left(\alpha^{2} L - \frac{\alpha}{2}\right) \lVert \mathfrak{g}_{\alpha}(x_{k})\rVert_{2}^{2}
      - \frac{1}{2 \alpha} \lVert x_{k+1} - \overline x_{k+1}\rVert_{2}^{2}
            \\ \leq{}&
      h(x_{k}) + \frac{\alpha}{2}\lVert \nabla f(x) - g_{k}\rVert_{2}^{2}
      + \left(\frac{L}{2} - \frac{1}{2 \alpha}\right) \lVert x_{k+1} - x_{k}\rVert_{2}^{2} 
             + \left(\alpha^{2} L - \frac{\alpha}{2}\right) \lVert \mathfrak{g}_{\alpha}(x_{k})\rVert_{2}^{2}.
  \end{align*}
  The proof is completed by taking the conditional expectation on both sides.
\end{proof}

 \end{document}